\documentclass[12pt]{amsart}
 
\usepackage{amsfonts,latexsym,rawfonts,amsmath,amssymb,amsthm,mathrsfs}
\usepackage[plainpages=false]{hyperref}
\usepackage{graphicx}
\usepackage{comment}
\usepackage[pagewise]{lineno}
 
\usepackage{xcolor}

\numberwithin{equation}{section}

\RequirePackage{color}
\theoremstyle{plain}
\newtheorem{theorem}{Theorem}[section]
\newtheorem{lem}[theorem]{Lemma}
 
\newtheorem{cor}[theorem]{Corollary}

\newtheorem{definition}[theorem]{Definition}

\newtheorem{proposition}{Proposition}[section]
\newtheorem{problem}{Problem}[section]

\newcommand{\beq}{\begin{equation}}
\newcommand{\eeq}{\end{equation}}
\newcommand{\beqs}{\begin{eqnarray*}}
\newcommand{\eeqs}{\end{eqnarray*}}
\newcommand{\beqn}{\begin{eqnarray}}
\newcommand{\eeqn}{\end{eqnarray}}
\newcommand{\beqa}{\begin{array}}
\newcommand{\eeqa}{\end{array}}

\def\S{\mathbb S}
\def\R{\mathbb R}

\def\M{\mathcal M}

\begin{document}

\title{A flow approach to the Musielak-Orlicz-Gauss image problem}

\author{Qi-Rui Li, Weimin Sheng, Deping Ye and Caihong Yi} 

	
\date{} 
	
\begin{abstract}
In this paper, the extended Musielak-Orlicz-Gauss image problem is studied. Such a problem aims to characterize the Musielak-Orlicz-Gauss image measure $\widetilde{C}_{G,\Psi,\lambda}(\Omega,\cdot)$ of convex body $\Omega$ in $\mathbb{R}^{n+1}$ containing the origin (but the origin is not necessary in its interior). In particular, we provide solutions to the extended Musielak-Orlicz-Gauss image problem based on the study of suitably designed parabolic flows, and by the use of approximation technique (for general measures). Our parabolic flows involve two Musielak-Orlicz functions and hence contain many well-studied curvature flows related to Minkowski type problems as special cases.  Our results not only generalize many previously known solutions to the Minkowski type and Gauss image problems, but also provide solutions to those problems in many unsolved cases.  


\end{abstract}
	
	\maketitle

\section{Introduction and overview of the main results} \label{introduction--1} \setcounter{equation}{0}

  The seminal Minkowski problem, dated back to  \cite{min1897, min1903},  has great influences on convex geometry and partial differential equations. It asks whether there exists a convex body whose surface area measure equals to the pregiven nonzero finite Borel measure on the unit sphere $\S^n$. The Aleksandrov problem is another classical example of Minkowski type problems; it aims to characterize the Aleksandrov integral curvature \cite{Alex42}. In certain circumstances, to solve the Minkowski and Aleksandrov problems requires to find solutions to some fully nonlinear PDEs (say Monge-Amp\`ere type equations). We refer readers to e.g. \cite{Caf90, CY76, GL97, Nir53, P78} for solutions to the  Minkowski and Aleksandrov problems. 

The past few years have witnessed the great progress on the Minkowski type problems and new Minkowski type problems are continuously emerging, including for example  the $L_p$ Minkowski problem \cite{Lu93}, the Orlicz-Minkowski problem  \cite{HLYZ2010}, the dual Minkowski  \cite{HLYZ16}, the $L_{p}$  dual Minkowski  \cite{LYZ18},  the  dual Orlicz-Minkowski problems \cite{GHWXY18, GHXY18, XY2017-1, ZSY2017}, and  the ($L_p$ and Orlicz) Aleksandrov problem \cite{FH, HLYZ-Alex}. There is a growing body work to these problems and we only name a few \cite{BHP18, BLDZ13, BLYZZ17,CHZ18, CL18, CLZ17, CLZ19, ChWang06, HP18, HZ18, HLYZ05, LW13, LO95, Zhao17, Zhao18}. The list is far from complete and more can be found in the references in, e.g.,  \cite{ HSYZ-21, HLYZ16, HLYZ-Alex, ZSY2017}.  On the one hand,  these Minkowski type problems can be reformulated by  Monge-Amp\`ere type equations (assume enough smoothness) and hence  largely enrich the theory of fully nonlinear PDEs.  On the other hand, these Minkowski type problems greatly push forward the development of the Brunn-Minkowski theory of convex bodies. For example, the $L_p$ Brunn-Minkowski theory was revived by Lutwak's groundbreaking work \cite{Lu93}, where the $L_p$ addition of convex bodies \cite{Firey-62} was combined with the volume to obtain the variational formula deriving the $L_p$ surface area measure (the central object in the $L_p$ Minkowski problem). The development of the Orlicz-Brunn-Minkowski theory  (see e.g., \cite{GHW14, Lud10, LYZ10 Adv, LYZ10 JDG,  XJL,ZHY2016}) also owes greatly to the Orlicz-Minkowski problem  \cite{HLYZ2010}. 

A recent breakthrough by B\"{o}r\"{o}czky {\it et.~al.} \cite{BLYZZ20} is the Gauss image problem, which links two given Borel measures $\lambda$ and $\mu$ on $\S^n$ via the radial Gauss image $\alpha_{\Omega}$ of a convex body $\Omega$. It asks: {\em 
under what conditions on $\lambda$ and $\mu$, does there exist a convex body $\Omega$ 
such that $\mu=\lambda(\alpha_{\Omega}(\cdot))$?} The existence and uniqueness of the solution
under the condition that $\lambda$ and $\mu$ are Aleksandrov related were proved in \cite{BLYZZ20}.
If the measures have density functions, Li and Wang \cite{LW18} obtained the same results from the optimal transportation viewpoint. We would like to mention that, although in certain circumstance the Gauss image problem becomes the Minkowski type problem, their difference is apparent: the former one involves two pre-given measures, while the latter one involves one pre-given measure.   
 
One of the main contributions of the general dual Orlicz-Minkowski problem \cite{GHWXY18, GHXY18} is the introduction of the Musielak-Orlicz function into the Minkowski type problem; and the Musielak-Orlicz function was used to define the general dual volume, a replacement of the volume and the $q$th dual volume. Throughout this paper, a Musielak-Orlicz function is a function $G: (0,\infty)\times \mathbb{S}^n\to \mathbb{R}$ such that both $G$ and $G_z(z, \xi)=\partial_z G(z,\xi)$ are continuous on $(0,\infty)\times\S^n$, where $G_z$ denotes the first order partial derivative of $G$ with respect to its first variable (see e.g., \cite{Harjulehto, Musielak} for more details).  The Musielak-Orlicz functions play important roles in analysis  \cite{Harjulehto, Musielak}, and introducing these functions into the Brunn-Minkowski theory will naturally lead to a new generation of the Brunn-Minkowski theory, namely 
the Musielak-Orlicz-Brunn-Minkowski theory of convex bodies.  A first step toward to this new theory has been done in Huang {\it et.\ al.}  \cite{HSYZ-21},
 where the Musielak-Orlicz addition was defined (see \eqref{genplus21}) and a variational formula of such an addition in terms of the general dual volume with respect to a finite Borel measure $\lambda$ on $\S^n$ was obtained (see \eqref{variation-11-27-12}). Moreover, the Musielak-Orlicz-Gauss image problem  aiming to characterize  the Musielak-Orlicz-Gauss image measure $\widetilde{C}_{\Theta}(\Omega,\cdot)$ for convex body $\Omega$ in $\R^{n+1}$ has been introduced in  \cite{HSYZ-21}. Here, $\Theta=(G,\Psi,\lambda)$ is a given triple with two Musielak-Orlicz functions $G$ and $\Psi$ defined on $(0,+\infty)\times\S^n$, and $\lambda$ a nonzero finite Borel measure on $\S^n$. Under the condition that $G$ is decreasing on its first variable, the existence of solutions to the Musielak-Orlicz-Gauss image problem is established in \cite{HSYZ-21}. See \cite{WuWuXiang} for a special case, the $L_p$ Gauss image problem and its solutions.

The major goal of this paper is to solve the (extended) Musielak-Orlicz-Gauss image problem under the condition that $G$ is increasing on its first variable. Our approach is based on the study of suitably designed parabolic flows. The flow technique has been proved to be effective and powerful in solving the Minkowski type and Gauss image problems  \cite{BIS16, Bryanivakisch, CLLN20, ChenTWX, ChenWX, ChWang00, LSW16, LL20,  LiuLu1, SY20}. The idea behind the flow technique is the fact that  the Minkowski type and Gauss image problems can be reformulated as a Monge-Amp\`ere type equation on $\mathbb{S}^n$, and this  indeed works for the (extended) Musielak-Orlicz-Gauss image problem due to equation \eqref{elliptic} (see \cite{HSYZ-21}).

  We say that $\Omega\subseteq \mathbb{R}^{n+1}$ is a convex body if it is a compact convex set with nonempty interior. 
Denote by $\mathcal{K}$ the set of all convex bodies in $\mathbb{R}^{n+1}$ containing the origin.
Let $\mathcal{K}_0\subseteq \mathcal{K}$ be the set of all convex bodies with the origin in their interiors. For $\Omega\in\mathcal{K}$,  define its radial function $r_\Omega: \S^n\rightarrow [0, \infty)$ and support function $u_\Omega: \S^n\rightarrow [0, \infty)$, respectively, by  
\begin{eqnarray}  \label{def u} \ \ \ \ 
r_\Omega(x)=\max\{a \in \mathbb{R}: a x \in \Omega\} \ \ \mathrm{and} \ \ u_\Omega(x)=\max\{\langle x, y\rangle, \,y\in \Omega\}, \, \ \,x\in\S^n, 
\end{eqnarray} where $\langle x, y\rangle$ denotes the inner product in $\mathbb{R}^{n+1}$.

For $\Omega\in\mathcal{K}$, let $\partial \Omega$ be its boundary. The Gauss map  of $\partial \Omega$, denoted by $\nu_{\Omega}: \partial \Omega\to\S^n$, is defined as follows: for $y\in \partial \Omega$, 
\begin{eqnarray*}
\nu_{\Omega}(y)=\{x\in\S^n: \langle x, y\rangle=u_\Omega(x)\}.
\end{eqnarray*} Let $\nu_{\Omega}^{-1}: \S^n\rightarrow \partial \Omega$ be the reverse Gauss map such that 
\begin{eqnarray*}
\nu_{\Omega}^{-1}(x)=\{y\in\partial \Omega: \langle x, y\rangle=u_\Omega(x)\}, \ \ \ x\in \S^{n}.
\end{eqnarray*}  Denote by $\alpha_\Omega: \S^n\rightarrow \S^n$   the radial Gauss image of $\Omega$. That is,  $$\alpha_{\Omega}(\xi)=\{x\in \S^n:  x\in \nu_{\Omega}(r_{\Omega}(\xi)\xi)\}, \ \ \xi\in \S^n. $$ Define $\alpha^*_{\Omega}: \S^n\rightarrow \S^n$, the reverse radial Gauss image of $\Omega$ as follows: for any Borel set $E\subseteq\S^n$,  
\begin{equation}\label{rev-rad-gauss} \alpha^*_\Omega(E)=\{\xi\in\S^n: r_\Omega(\xi)\xi\in\nu_\Omega^{-1}(E)\}.\end{equation} 
We often omit the subscript $\Omega$ in $r_{\Omega}$, $u_{\Omega}$, $\nu_{\Omega}$, $\nu_{\Omega}^{-1}$, $\alpha_{\Omega}$,  and $\alpha_{\Omega}^*$ if no confusion occurs.

Let $\mathcal{C}$ be the set of all Musielak-Orlicz function $G: (0,\infty)\times \mathbb{S}^n\to \mathbb{R}$ such that both $G$ and $G_z(z, \xi)=\partial_z G(z,\xi)$ are continuous on $(0,\infty)\times\S^n$, where $G_z$ denotes the first order partial derivative of $G$ with respect to its first variable.  Let $\mathcal{C}_I\subseteq \mathcal{C}$ be the set of all $G\in \mathcal{C}$ such that $G_z>0$ on $(0,\infty)\times \mathbb{S}^n$. Similarly, $\mathcal{C}_d\subseteq\mathcal{C}$ is the set consisting of all $G\in \mathcal{C}$ such that $G_z<0$ on  $(0,\infty)\times \mathbb{S}^n$.  By $d\xi$, we mean the usual spherical measure on $\S^n$.  
For $\Omega\in \mathcal{K}$ a convex body containing the origin $o\in \mathbb{R}^{n+1}$, let $N(\Omega, o)$  be the normal cone of $\Omega$ at the origin $o$,  namely, $N(\Omega, o)$ is the closed convex cone defined by    $$N(\Omega, o)=\big\{y\in \mathbb{R}^{n+1}: \langle \tilde{y}, y\rangle\leq 0 \ \ \mathrm{for}\ \tilde{y}\in \Omega\big\}.$$ Clearly, $N(\Omega, o)=\{o\}$ when $\Omega\in \mathcal{K}_0$. The Musielak-Orlicz-Gauss image measure of $\Omega\in \mathcal{K}$ is defined as follows. 

 \begin{definition}\label{def-unified}
Let $\Theta=(G,\Psi,\lambda$) be a given triple such that $G\in \mathcal{C}$, $\Psi\in \mathcal{C}_I\cup\mathcal{C}_d$, and  $\lambda$ is a nonzero finite Borel measure on $\S^n$. Denote by $\psi=z\Psi_z$.  Let  $\widetilde{C}_{\Theta}(\Omega,\cdot)$ be  the Musielak-Orlicz-Gauss image measure  of $\Omega\in\mathcal{K}$, which is defined  as follows: for each Borel set $\omega\subseteq\S^n$, 
\begin{eqnarray}
\widetilde{C}_{\Theta}(\Omega,\omega)= \int_{{\mathrm{\alpha}_{\Omega}^*}(\omega\setminus N(\Omega, o))} \frac{r_\Omega(\xi)G_z(r_\Omega(\xi),\xi)}{\psi(u_\Omega(\alpha_\Omega(\xi)),\alpha_\Omega(\xi))}\,d\lambda(\xi), \label{GDOCM def}
\end{eqnarray} where, in addition,  $G$ and $\Psi$ are assumed to be Musielak-Orlicz functions defined on $[0, \infty)\times \S^n$  if $\Omega\in \mathcal{K} \setminus \mathcal{K}_0$ is a convex body with the origin in its boundary. 
 \end{definition} When $\Omega\in \mathcal{K}_0$, this recovers  Definition 3.1 in \cite{HSYZ-21}.  
For convenience, let $\widetilde{C}_{G,\lambda}(\Omega, \cdot)=\widetilde{C}_{(G, \log t, \lambda)}(\Omega,\cdot)$, (in this case, $\psi\equiv1$). That is, for each Borel set $\omega\subseteq\S^n$, 
\begin{eqnarray}
\widetilde{C}_{G,\lambda}(\Omega,\omega) = \int_{{\mathrm{\alpha}_{\Omega}^*}(\omega\setminus N(\Omega, o))} r_\Omega(\xi)G_z(r_\Omega(\xi),\xi)\,d\lambda(\xi). \label{GDOCM def-log}
\end{eqnarray} 
It can be checked that, for $\Omega\in \mathcal{K}_0$, \begin{equation} 
\frac{\,d\widetilde{C}_{\Theta}(\Omega,\cdot)}{\,d \widetilde{C}_{G,\lambda}(\Omega, \cdot)}=\frac{1}{\psi(u_{\Omega}(\cdot), \cdot)}. \label{R-N-expression}
\end{equation}  It has been proved in \cite{HSYZ-21} that $\widetilde{C}_{\Theta}(\Omega,\cdot)$ for $\Omega\in \mathcal{K}_0$ is a finite signed Borel measure on $\S^n$. Moreover,  $\widetilde{C}_{\Theta}(\Omega, \cdot)$ for  $\Omega\in \mathcal{K}_0$  arises from calculating the variation of $\widetilde{V}_{G,\lambda}(\Omega)$ in terms of the Musielak-Orlicz addition defined in \eqref{genplus21}. Here, $\widetilde{V}_{G,\lambda}(\Omega)$ is the general dual volume of $\Omega\in \mathcal{K}_0$ with respect to $\lambda$ defined by 
 \begin{eqnarray}\label{VG def}
\widetilde{V}_{G,\lambda}(\Omega)=\int_{\mathbb{S}^n}G(r_\Omega(\xi),\xi)d\lambda(\xi),
\end{eqnarray}  and   the {\em Musielak-Orlicz
addition} is formulated by 
\begin{align}\label{genplus21} \Psi(f_{\varepsilon}(\xi), \xi)= \Psi(f(\xi), \xi)+\varepsilon g(\xi) 
\end{align}  where $f:\S^n\rightarrow (0, \infty)$ is a positive function, $g: \S^n\rightarrow \mathbb{R}$ is a function,  and $\varepsilon\in (-\varepsilon_0, \varepsilon_0)$ for small enough $\varepsilon_0>0$.  More precisely,  for $\Upsilon\subseteq \S^n$ a closed
set that is not contained in any closed hemisphere of $\S^n$, $\Theta=(G, \Psi, \lambda)$ a triple such that $G\in \mathcal{C}$,  $\Psi \in \mathcal{C}_I\cup\mathcal{C}_d$ and $\lambda$ is a nonzero finite Borel measure on $\S^n$  absolutely continuous w.r.t.~$d\xi$, one has  
 \begin{align}  \label{variation-11-27-12}
    \lim_{\varepsilon\rightarrow 0}\frac{\widetilde{V}_{G,\lambda}([f_{\varepsilon}])-\widetilde{V}_{G,\lambda}([f])}{\varepsilon}=
    \int_{\Upsilon} g(u)\, d\widetilde{C}_{\Theta}([f], u),
    \end{align} where $[f]$ denotes the
Wulff shape generated by  $f$  \begin{align*} 
[f]= \bigcap_{\xi\in\Omega}\big\{x\in\R^{n+1}: \langle x, \xi\rangle  \leq f(\xi)\big\}. 
\end{align*}

The following Musielak-Orlicz-Gauss problem has been posed in \cite{HSYZ-21}:
\textsl{Let $\Theta=(G, \Psi, \lambda)$ be a given triple such that $G\in \mathcal{C}, \Psi\in \mathcal{C}$ and $\lambda$ is a nonzero finite Borel measure on $\S^n$. Under what conditions on the triple  $\Theta$ and a nonzero finite Borel measure $\mu$ on $\S^n$ do there exist a $\Omega\in \mathcal{K}_0$ (ideally) and a constant $\tau\in\mathbb{R}$ such that 
\begin{eqnarray}\label{GDOMP}
\,d\mu = \tau \,d\widetilde{C}_{\Theta}(\Omega,\cdot)?
\end{eqnarray} } In view of \eqref{R-N-expression},  one can rewrite \eqref{GDOMP}  as
\begin{eqnarray}\label{GDOMP1}
\psi(u_\Omega(\cdot), \cdot)\,d\mu = \tau \,d\widetilde{C}_{G,\lambda}(\Omega,\cdot).
\end{eqnarray}

Under the condition that $G$ is decreasing on its first variable, the existence of solutions to the Musielak-Orlicz-Gauss image problem is established in \cite{HSYZ-21}.
Although \eqref{GDOMP} and \eqref{GDOMP1} are equivalent when $\Omega\in \mathcal{K}_0$, the latter one arguably has better features and in particular allows the extension of the  Musielak-Orlicz-Gauss problem to  $\Omega\in \mathcal{K}$. Such an extension is extremely important when $G$ is increasing in its first variable, as one can see in many results in the literature, such as \cite{BF19,CL19,GHXY18}. In this paper, our primary goal  is to study the following extended Musielak-Orlicz-Gauss problem (in particular, when $G$ is increasing in its first variable). 
\begin{problem}  \label{extended-MOGP} Let $\Theta=(G, \Psi, \lambda)$ be a given triple where $G: [0, \infty)\times \S^n\rightarrow [0, \infty)$ and $\Psi:  [0, \infty)\times \S^n\rightarrow [0, \infty)$ are two continuous functions, and $\lambda$ is a nonzero finite Borel measure on $\S^n$. Under what conditions on the triple  $\Theta$ and a nonzero finite Borel measure $\mu$ on $\S^n$ do there exist a $\Omega\in \mathcal{K}$ and a constant $\tau\in\mathbb{R}$ such that  \eqref{GDOMP1} holds. \end{problem}

We would like to provide some intuitions and pictures behind the introduction of the  Musielak-Orlicz functions into the areas of parabolic flows or convex geometry, especially for the Musielak-Orlicz-Gauss image problem. First of all, the well-known additions in literatures, such as the $L_p$ and the Orlicz additions \cite{Firey-62, GHW14, XJL}, of convex bodies are ``uniform" on directions $\xi\in \S^n,$ in the sense that 
the functions $t^p$ for the $L_p$ addition ($0\neq p\in \mathbb{R}$) and $\varphi(t)$ for the Orlicz addition, for $t\in (0, \infty)$, are independent of directions. This is not the case for the Musielak-Orlicz additions of convex bodies, and hence the Musielak-Orlicz
additions of convex bodies greatly enrich the algebraic structures and analytic aspects  on the set of convex bodies. We believe that these will bring new interesting results and inequalities in convex geometry, for example the Musielak-Orlicz analogues of the Orlicz-Brunn-Minkowski, Orlicz centroid and Orlicz projection inequalities \cite{GHW14, LYZ10 Adv, LYZ10 JDG, XJL}. Secondly, the variational formula \eqref{variation-11-27-12} provides a geometric intuition for the measure $\widetilde{C}_{\Theta}(\Omega,\cdot)$. That is, when extending the convex body $\Omega  \in \mathcal{K}_0$ based on the Musielak-Orlicz addition \eqref{genplus21},  the measure $\widetilde{C}_{\Theta}(\Omega,\cdot)$ serves as ``the base area" to estimate $
\widetilde{V}_{G,\lambda}([(u_{\Omega})_\varepsilon]\setminus \Omega)$  for  $\varepsilon>0$ small enough where $(u_{\Omega})_\varepsilon$ is given by \begin{align*}  \Psi\big((u_{\Omega})_{\varepsilon}(\xi), \xi\big)= \Psi\big(u_{\Omega}(\xi), \xi\big)+\varepsilon g(\xi).\end{align*}
 This coincides with the geometric meanings of the surface area and its recent extensions in, e.g., \cite{BLYZZ20, GHWXY18, HLYZ2010, HLYZ16, Lu93, LYZ18}. On the other hand, it has been proved in \cite[(3.3)]{HSYZ-21} that, for $\Omega\in \mathcal{K}_0$, 
      \begin{align*}  \frac{\,d \widetilde{C}_{\Theta}(\Omega,\xi)} {\,d\lambda^*(\Omega, \xi)} = \frac{\rho_{\Omega} (\alpha^*_{\Omega}(\xi)) G_z(\rho_{\Omega}(\alpha^*_{\Omega}(\xi)), \alpha^*_{\Omega}(\xi))}{u_{\Omega}(\xi)\Psi_z(u_{\Omega}(\xi), \xi)}\ \ \ \mathrm{for}\ \ \xi\in \S^n,\end{align*} where $\lambda^*(\Omega, \cdot)= \widetilde{C}_{(\log t, \log t, \lambda)}(\Omega,\cdot)$ is the Gauss image measure \cite{BLYZZ20}. Analytically, the measure $\widetilde{C}_{\Theta}(\Omega,\cdot)$ is absolutely continuous with respect to $\lambda^*(\Omega, \cdot)$ and can be viewed as a weighted Gauss image measure (in arguably the most general way). Moreover, the Musielak-Orlicz-Gauss image problem integrates the Minkowski type problems  \cite{ FH, GHWXY18, GHXY18,HLYZ2010, HLYZ16, HLYZ-Alex,   Lu93,  LYZ18,  XY2017-1, ZSY2017}  and the Gauss image problem \cite{BLYZZ20} into a unified formula. This unification will provide a uniform way to deal with the problem to characterize measures derived from variational formulas based on the algebraic combinations of convex bodies, and in particular will help to advance the recent development of the measure theoretical Brunn-Minkowski theory of convex bodies.      
Last but not the least, the Musielak-Orlicz functions bring extra ingredients to parabolic flows and partial differential equations through the Monge-Amp\`ere type equation related to the  Musielak-Orlicz-Gauss image problem (see \eqref{elliptic}).  

Our main result can be summarized in the following theorem.  For convenience, let $$\widetilde{C}_{G,\lambda}(\Omega,\S^n)=\int_{\S^n}\,d\widetilde{C}_{G,\lambda}(\Omega,\xi).$$ Let $\mathcal{G}_I^0$ be the class of continuous functions $G: [0,\infty)\times \mathbb{S}^n\to [0,\infty)$ such that 
\begin{itemize}
\item $zG_z(z,\xi)$ is continuous on $[0,\infty)\times \mathbb{S}^n$;
\item $G_z > 0$ on $(0,\infty)\times \mathbb{S}^n$;
\item $G(0, \xi)=0$ and $zG_z(z,\xi) = 0$ at $z = 0$ for $\xi\in \mathbb{S}^n$. 
\end{itemize}
Clearly, if $G\in \mathcal{G}_I^0$, for any $\xi\in \S^n$, then $G(z, \xi)$ is strictly increasing on $z\in (0, \infty)$.  Again, we write $\psi: [0,\infty)\times\S^n\to [0, \infty)$ for  the function $\psi=z\Psi_z$. In particular, $\psi(0, \xi)=\lim_{z\rightarrow 0^+} \psi(z, \xi)$ for each $\xi \in \S^n$. 
 
\begin{theorem}\label{main1}
 Let $G \in \mathcal{G}_I^0$, $\Psi\in \mathcal{G}_I^0$ and $\lambda$ be a nonzero finite Borel measure on $\S^n$. Assume the following conditions on $G$, $\lambda$ and $\Psi$. \\
 $(\mathrm{i})$  $d\lambda(\xi)=p_\lambda(\xi)d\xi$ where the function $p_\lambda:\S^n\to(0,\infty)$ is
  continuous. \\
 $(\mathrm{ii})$  For all $x\in \S^n$, the following holds:  \begin{equation}\label{Psi-comp-condition} \lim_{s\rightarrow +\infty}\Psi(s, x)=+\infty. \end{equation}  
   
 Let $\mu$ be a nonzero finite Borel measure on $\mathbb{S}^n$ that is not concentrated on any closed hemisphere. Then there is a convex body $\Omega\in\mathcal{K}$ such that $\eqref{GDOMP1}$ holds, with the constant \begin{eqnarray*}\tau=
\frac{1}{ \widetilde{C}_{G,\lambda}(\Omega,\S^n)} \int_{\S^n}\psi(u_\Omega(x), x)\,d\mu(x).
\end{eqnarray*}  \end{theorem}
We would like to mention that Theorem \ref{main1} holds for  $G=t^q$ with $q>0$, $\Psi=t^p$ with $p>0$ and $\,d\lambda=\,d\xi$. This not only covers the results for the $L_p$ dual Minkowski problem for $p>1$ and $q>0$ by B\"{o}r\"{o}czky and Fodor \cite{BF19}, but also obtains  solutions to the unsolved case $0<p\leq 1$ and $q>0$. Similarly, when $\,d\lambda=\,d\xi$ and $\Psi(t, \xi)=\varphi(t)$ for all $(t, \xi)\in [0, \infty)\times \S^n$ is an Orlicz function, Theorem \ref{main1} covers the case in \cite[Theorem 6.3]{GHXY18}  but goes beyond, namely Theorem \ref{main1}  removes the condition that $\lim_{t\rightarrow 0^+} \varphi'(t)=\lim_{t\rightarrow 0^+} \Psi_z(t, \xi)=0$, a condition crucial in \cite[Theorem 6.3]{GHXY18}.  Our proof of Theorem \ref{main1} is based on the study of a suitably designed parabolic flow  and the use of approximation argument. The idea of using the parabolic flow comes from the fact that Problem \ref{extended-MOGP} can be rewritten as a Monge-Amp\`ere type equation on $\mathbb{S}^n$. Assume that $\lambda$ is a nonzero finite Borel measure on $\S^n$ and $\,d\lambda(\xi)=p_\lambda(\xi)\,d\xi$ with the function $p_\lambda:\S^n\to(0,\infty)$ being continuous.  When the pregiven measure $\mu$ has a density $f$ with respect to $\,d\xi$, \eqref{GDOMP}  reduces to solving the following Monge-Amp\`ere type equation on $\mathbb{S}^n$ (see \cite{HSYZ-21}):
\begin{eqnarray}\label{elliptic}\ \ \ \ \ \ \ \ 
u(u^2+|\nabla u|^2)^{-\frac{n}{2}}G_z(\sqrt{u^2+|\nabla u|^2},\xi)p_\lambda(\xi)\det(\nabla^2 u+uI)=\gamma f(x){\psi(u,x)},
\end{eqnarray}
where $\nabla$ and $\nabla^2$ are the gradient and Hessian operators with respect to an orthonormal frame on $\mathbb{S}^n$, $\gamma>0$ is a constant, $I$ is the identity matrix, and $\xi=\alpha_{\Omega_u}^*(x)$ where $\alpha_{\Omega_u}^*$ is the reverse radial Gauss map of $\Omega_u$ -- the convex body whose support function is $u(x)$ for $x\in \S^n$.

Let  $f:\mathbb{S}^n\rightarrow (0,\infty)$ be  a smooth positive function, and $\Omega_0\in \mathcal{K}_0$ be  a convex body such that  $\mathcal{M}_0=\partial\Omega_0$ is a smooth and uniformly convex hypersurface.  Equation \eqref{elliptic} suggests the following  curvature flow, by letting $r=\sqrt{u^2+|\nabla u|^2}$, 
\begin{equation}\label{SSF}
\left\{
\begin{array}{ll}
\frac{\partial{X}}{\partial{t}} (x,t)&=\left(-f(\nu){\psi} (u,x)r^{n}G_z(r,\xi)^{-1}p^{-1}_\lambda(\xi)K+\eta(t)u\right)\nu ,\\\\
X(x,0)&=X_0(x),
\end{array}
\right.
\end{equation}
where $X(\cdot,t): \mathbb{S}^n\to\mathbb{R}^{n+1}$ is  the embedding that parameterizes a family of convex hypersurfaces $\mathcal{M}_t$ (in particular, $X_0$ is the parametrizsation of $\mathcal{M}_0$), 
$\Omega_t$ is the convex body circumscribed by $\mathcal{M}_t$,  $K$ denotes the Gauss curvature of $ \Omega_t$ at $X(x,t)$, $\nu$ denotes the unit outer normal of $ \Omega_t$ at $X(x,t)$, $u$ is the support function of  $ \Omega_t$, $\xi=\alpha^*_{\Omega_t}(x)$, and
\begin{eqnarray}\label{eta def}
\eta(t)=\frac{\int_{\S^n}f {\psi}(u,x) dx}{\int_{\S^n}rG_z(r,\xi) p_\lambda(\xi)d\xi}.
\end{eqnarray} As \eqref{SSF} involves the Musielak-Orlicz functions, such a flow could be named a Musielak-Orlicz-Gauss curvature flow, which is arguably the most general curvature flow related to Minkowski type and Gauss image problems and contains all previous well-studied flows \cite{BIS16, Bryanivakisch, CLLN20, ChenTWX, ChenWX, ChWang00, LSW16, LL20, LiuLu1, SY20} as its special cases. 

Assume enough smoothness on the functions $f$, $p_\lambda$, $G$ and $\psi$. We show that,  under the condition that  \begin{equation}\label{good-condition-1}    
\liminf_{s\to0^+}\frac{sG_z(s,x)}{\psi(s,x)}=\infty, \ \ \mathrm{for\ all}\ x\in\S^n, 
\end{equation}   the flow \eqref{SSF} deforms a smooth and uniformly convex hypersurface to a limit hypersurface satisfying \eqref{GDOMP}. Establishing  a priori estimate is the key ingredient for such a convergence; moreover, the $C^0$ estimate in this case can be  obtained by the maximum principle. However, when the condition  \eqref{good-condition-1} is replaced by the following condition: 
\begin{equation}\label{good-condition-2}   
\liminf_{s\to 0^+}\frac{sG_z(s,x)}{\psi(s,x)}< \infty, \ \ \mathrm{for\ some}\, \ x\in\S^n, 
\end{equation}   the $C^0$ estimate {\it cannot}  be obtained by the maximum principle directly. In order to overcome this obstruction,  the flow \eqref{SSF} shall be replaced by a more carefully designed one with the function $\psi$ replaced by the smooth function $\widehat{\psi}_\varepsilon: [0,\infty)\times\S^n\rightarrow [0, \infty)$ as follows:
\begin{equation}\label{hatpsi eps}
\widehat{\psi}_{\varepsilon}(s,x)=\left\{
\begin{array}{ll}
\psi(s,x), &\textrm{if $s\ge2\varepsilon$},\\\\
G_z(s,\alpha^*(x))s^{1+\varepsilon},&\textrm{if $0\le s\le\varepsilon$},
\end{array}
\right.
\end{equation} and $\widehat{\psi}_\varepsilon(s,x)\leq C_0$ for  $(s, x)\in(\varepsilon, 2\varepsilon)\times \S^n$ is chosen  so that $\widehat{\psi}_{\varepsilon}$ is  smooth on $[0, \infty)\times \S^n$ and $\widehat{\psi}_\varepsilon(s,x)>0$ for all $(s,x)\in(0,\infty)\times\S^n$. Hereafter, 
 \begin{eqnarray} \label{constant-c-0} 
C_0=\max\{1, \max_{(s,x)\in[0,2]\times\S^n }\psi(s,x)\}
\end{eqnarray}  and $\varepsilon\in(0,\delta)$ satisfies \begin{equation*}   \max_{(s,\xi)\in [0, \varepsilon]\times\S^n}G_z(s,\xi)s^{1+\varepsilon}\le s^\varepsilon\le C_0,\end{equation*} where $\delta\in (0, 1)$ is a constant such that  \begin{eqnarray} \max_{(s,\xi)\in [0,\delta]\times\S^n} sG_z(s,\xi)\le 1\label{def-delta}\end{eqnarray} (the existence of $\delta$ is guaranteed if  $zG_z(z,\xi)$ is continuous on $[0,\infty)\times \S^n$ and $zG_z(z,\xi) = 0$ at $z = 0$ for $\xi\in \S^n$).   In other words, the following curvature flow is considered: \begin{equation}\label{SF2}
\left\{
\begin{array}{ll}
\frac{\partial{X_{\varepsilon}}}{\partial{t}} (x,t)&=\left(-f(\nu)\widehat{\psi}_\varepsilon(u_{\varepsilon}, x)r^{n}G_z(r,\xi)^{-1}p_\lambda^{-1}(\xi) K+\eta_\varepsilon(t)u_{\varepsilon}\right)\nu ,\\\\
X_{\varepsilon}(x,0)&=X_0(x),
\end{array}
\right.
\end{equation} 
where $X_\varepsilon(\cdot,t): \mathbb{S}^n\to\mathbb{R}^{n+1}$ parameterizes convex hypersurface
$\mathcal{M}_t^\varepsilon$, 
$u_{\varepsilon}$ denotes the support function of the convex body $\Omega^{\varepsilon}_t$
circumscribed by $\mathcal{M}_t^\varepsilon$, and 
\begin{eqnarray}\label{eta epi def}
\eta_\varepsilon(t)=\frac{\int_{\S^n}f\widehat{\psi}_\varepsilon(u,x) dx}{\int_{\S^n}rG_z(r,\xi) p_\lambda(\xi)d\xi}.
\end{eqnarray}

It will be proved in Lemma \ref{upperbound} that  $u_\varepsilon(\cdot,t)$ is uniformly bounded from above. A positive uniform lower bound estimate for  $u_\varepsilon(\cdot,t)$ will be given in Lemma \ref{lowerbound}, and this argument relies on the construction for $\widehat{\psi}_\varepsilon$ in \eqref{hatpsi eps}.   The $C^0$ estimates make it possible to further obtain the higher order estimates and to show that flow \eqref{SF2} exists for all time. These, together with Lemma \ref{monotone1},  imply the existence of a sequence of times $t_i\to\infty$ such that $u_\varepsilon(\cdot,t_i)$ converges to a positive and uniformly convex function $u_{\varepsilon,\infty}\in C^\infty(\mathbb{S}^n)$ solving the equation below
\begin{eqnarray}\label{elliptic eq}\ \ \ \ \ \ \ \ 
u(u^2+|\nabla u|^2)^{-\frac{n}{2}}G_z(\sqrt{u^2+|\nabla u|^2},\xi)p_\lambda(\xi)\det(\nabla^2u+uI)=\gamma_\varepsilon f(x){\widehat{\psi}_\varepsilon(u, x)},
\end{eqnarray}
where $\gamma_\varepsilon>0$ is a constant. Furthermore, we prove that there is a sequence of $\varepsilon_i\to0$ such that $u_{\varepsilon_i,\infty}$ locally uniformly converges to a weak solution of \eqref{elliptic}. Throughout this paper, we say that $u\in C^2(\S^n)$, the set of functions on $\S^n$ with continuous second order derivatives,  is uniformly convex if the matrix $\nabla^2 u+uI$ is positively definite.  Thus, the following theorem can be obtained, which provides solutions to \eqref{elliptic} (and hence to the extended Musielak-Orlicz-Gauss problem) when $d\mu=fd\xi$.

\begin{theorem}\label{main3}  Let  $G\in \mathcal{G}_I^0$ be a smooth function. 
Suppose that $\,d\mu(\xi)=f(\xi)\,d\xi$ and $\,d\lambda(\xi)=p_\lambda(\xi)\,d\xi$ with $f$ and $p_\lambda$ being smooth and strictly positive on $\mathbb{S}^n$. Let  $\Psi\in \mathcal{G}_I^0$ be a smooth function satisfying \eqref{Psi-comp-condition}.
  The following statements hold.\\
$(\mathrm{i})$ If $G$ and $\psi$ satisfy \eqref{good-condition-1}, then one can find an $\Omega\in\mathcal{K}_0$  such that \eqref{GDOMP} holds;\\
$(\mathrm{ii})$ If $G$ and $\psi$ satisfy \eqref{good-condition-2},  then one can find an  $\Omega\in\mathcal{K}$ such that \eqref{GDOMP1} holds.
\end{theorem}

For a general measure $\mu$,  there is a sequence of measures $\{\mu_i\}_{i\in \mathbb{N}}$,
where $\,d\mu_i =f_i d\xi$ with $f_i$ being smooth and strictly positive on $\mathbb{S}^n$,
such that $\mu_i$ converges to $\mu$ weakly.
Theorem \ref{main1} is then proved by the virtue of Theorem \ref{main3} and an approximation argument.

This paper is organized as follows. In Section \ref{section-2},
some properties of convex hypersurfaces, and the flows \eqref{SSF} and \eqref{SF2} are presented.
In particular we show the strict monotonicity of functionals \eqref{function} and \eqref{functional} (see Lemmas \ref{mono} and \ref{monotone1}, respectively) and the preservation of $\widetilde{V}_{G,\lambda}(\cdot)$ along the flows (see Lemmas \ref{Volume} and  \ref{monotone1}). 
The $C^0$ estimates for the flows \eqref{SSF} and \eqref{SF2} when smooth functions $G$ and $\Psi$ satisfy the assumptions in Theorem \ref{main3} are established in Section \ref{section-3}.  
Section \ref{section-4} is dedicated to the proofs of the long time existence of the flows \eqref{SSF} and \eqref{SF2} and the proof of Theorem \ref{main3} as well.  Moreover, the existence of solutions to Problem \ref{extended-MOGP},  i.e., Theorem \ref{main1},  is proved by an approximation argument for general $\mu$ that is not concentrated on any closed hemisphere.
Section \ref{section-5} provides the second order derivative estimates which will be used in the study of flows \eqref{SSF} and \eqref{SF2}.

\section{Preliminary and properties of the flows}\label{section-2}
Let us recall some basic notations. In $\mathbb{R}^{n+1}$, let $\langle x, y\rangle $ be the inner product of $x, y\in \mathbb{R}^{n+1}$.  By $|x|$, we mean the Euclidean norm of $x\in \mathbb{R}^{n+1}$. The unit sphere $\S^n$ will be assumed to have a smooth local orthonormal frame field  $\{e_1, \cdots, e_n\}$. By $\nabla$ and $\nabla^2$, we mean the gradient and Hessian operators with respect to $\{e_1, \cdots, e_n\}$. The surface area of $\S^n$ is denoted by $|\S^n|$.  

Let  $\Omega\in \mathcal{K}$ be a convex body containing the origin. Denote by $w_\Omega:\mathbb{S}^n\to\mathbb{R}$  the width function of $\Omega$, which can be formulated by, for $x\in \S^n$,   \[w_\Omega(x)=u_\Omega(x)+u_\Omega(-x).\] We shall need $w_\Omega^+$ and $w_\Omega^-$,  the maximal and  minimal widths of $\Omega$, respectively, which can be formulated by
\begin{eqnarray}\label{mmwide}
w_\Omega^+=\max_{x\in\mathbb{S}^n}\{u_\Omega(x)+u_\Omega(-x)\}\ \ \mathrm{and}\ \ 
w_\Omega^-=\min_{x\in\mathbb{S}^n}\{u_\Omega(x)+u_\Omega(-x)\}.
\end{eqnarray} 
The following result is needed and can be found in  e.g. \cite[Lemma 2.6]{CL19}.
\begin{lem}\label{pro}
Let $\Omega\in\mathcal{K}_0$ be a convex body containing the origin in its interior. Let $u_{\Omega}$ and $r_{\Omega}$ be the support and radial functions of $\Omega$, and $x_{max}\in \S^n$ and $\xi_{min}\in \S^n$ be such that $u_{\Omega}(x_{max})=\max_{x\in\S^n}u_{\Omega}(x)$ and $r_{\Omega}(\xi_{min})=\min_{\xi\in\S^n}r_{\Omega}(\xi)$. Then
\begin{eqnarray*}
\max_{x\in \S^n}u_{\Omega}(x)&=&\max_{\xi\in \S^n}r_{\Omega}(\xi)\ \ \mathrm{and}\ \ \min_{x\in\S^n}u_{\Omega}(x)=\min_{\xi\in \S^n}r_{\Omega}(\xi),\\
u_{\Omega}(x)&\ge& \langle x,  x_{max}\rangle u_{\Omega}(x_{max}) \ \ \ \mathrm{for\ all}\ x\in\S^n,\\
r_{\Omega}(\xi) \langle\xi, \xi_{min}\rangle &\le&  r_{\Omega}(\xi_{min})\ \ \ \mathrm{for\ all}\ \xi\in\S^n.
\end{eqnarray*}
\end{lem}
 
 Let $\mathcal{M}$ be a smooth, closed, uniformly convex hypersurface in $\mathbb{R}^{n+1}$, enclosing the origin.  The parametrization of $\mathcal{M}$ is given by the inverse Gauss map $X:\mathbb{S}^n\to \mathcal{M}\subseteq \mathbb{R}^{n+1}$. It follows from \eqref{def u} and \eqref{rev-rad-gauss} that 
 \begin{eqnarray} 
X(x)=r(\alpha^*(x))\alpha^*(x)\ \ \mathrm{and}\ \ u(x)&=&\langle x,X(x)\rangle, 
 \label{support} \end{eqnarray} where  $u$ is the support function of (the convex body circumscribed by) $\mathcal{M}$.
It is well known that the Gauss curvature of $\mathcal{M}$ is \begin{equation}
K=\frac{1}{\det(\nabla^2 u+uI)},\label{curvature-formula-1}
\end{equation} and  the principal curvature radii of $\mathcal{M}$  are the eigenvalues of the matrix $
b_{ij}=\nabla_{ij}u+u\delta_{ij}.$  Moreover, the following hold, see e.g. \cite{LSW16},
\begin{eqnarray}  \label{nablau} 
 r\xi = u x +\nabla u, \ \  r=\sqrt {u^2+|\nabla u|^2},  \ \ \mathrm{and} \ \ u&=&\frac{r^2}{\sqrt{r^2+|{\nabla}r|^2}}.
\end{eqnarray}

Let $u(\cdot, t)$ and $r(\cdot, t)$ be the support and radial functions of $\mathcal{M}_t$. Recall that (see e.g., \cite[Lemma 2.1]{CCL19} or \cite{HLYZ16, LSW16})
\begin{eqnarray}\label{r--u}
 \frac{\partial_t r(\xi,t)}{r }=\frac{\partial_t u (x,t)}{u}.
\end{eqnarray} 
By \eqref{support} and \eqref{curvature-formula-1}, the flow equation \eqref{SSF} for  $\mathcal{M}_t$ can be reformulated by its support function $u(x,t)$ as follows: 
\begin{equation}\label{uSSF}
\left\{
\begin{array}{rl}
\partial_tu(x,t)&=-f(x){\psi}(u,x)r^{n}G_z(r,\xi)^{-1}p^{-1}_\lambda(\xi) K+\eta(t)u,\\\\
u(\cdot,0)&=u_0.
\end{array}\right.
\end{equation} It follows from \eqref{r--u} that the flow equation \eqref{SSF} for  $\mathcal{M}_t$ can be reformulated by its  radial function $r(\xi,t)$ as follows: 
\begin{equation}\label{rSSF}
\left\{
\begin{array}{rl}
\partial_tr(\xi,t)&=-f(x) {\psi} (u, x) u^{-1} r^{n+1}G_z(r,\xi)^{-1}p^{-1}_\lambda(\xi)K+\eta(t)r ,\\\\
r(\cdot,0)&=r_0.
\end{array}\right.
\end{equation} Similarly, $\mathcal{M}_t^\varepsilon$ is evolved by \eqref{SF2} if its support function $u(x,t)=u_{\varepsilon}(x,t)$ satisfies
\begin{equation}\label{uSF2}
\left\{
\begin{array}{rl}
\partial_tu(x,t)&=-f(x)\widehat{\psi}_\varepsilon(u,x)r^{n}G_z(r,\xi)^{-1}p^{-1}_\lambda(\xi) K+\eta_\varepsilon(t)u,\\\\
u(\cdot,0)&=u_0;
\end{array}\right.
\end{equation} or equivalently if its radial function $r(\xi,t)=r_{\varepsilon} (\xi,t)$ satisfies
\begin{equation*}
\left\{
\begin{array}{rl}
\partial_tr(\xi,t)&=-f(x) \widehat{\psi}_\varepsilon(u, x) u^{-1} r^{n+1}G_z(r,\xi)^{-1}p^{-1}_\lambda(\xi)K+\eta_\varepsilon(t)r ,\\\\
r(\cdot,0)&=r_0.
\end{array}\right.
\end{equation*}

It is well known that $J(\xi)$, the determinant of the Jacobian of the radial Gauss image $x=\alpha_{\Omega}(\xi)$ for $\Omega\in \mathcal{K}_0$, satisfies  (see e.g., \cite{CCL19, HLYZ16, LSW16})
 \begin{eqnarray} \label{J(xi)}
 J(\xi)=\frac{r_{\Omega}^{n+1}(\xi) K(r_{\Omega}(\xi)\xi)}{u_{\Omega}(\alpha_{\Omega}(\xi))}.\label{Jac-1-1} \end{eqnarray} 
Letting $x=\alpha_{\Omega}(\xi)$ (hence $\xi=\alpha_{\Omega}^*(x)$) and  $r_{\Omega}=r_{\Omega}(\xi)=r_{\Omega}(\alpha_{\Omega}^*(x))$, one has
\begin{eqnarray} \int_{\S^n} r_{\Omega}^{-n} G_z(r_{\Omega},\xi)p_\lambda (\xi)uK^{-1}\,dx=\int_{\mathbb{S}^n} r_{\Omega}G_z(r_{\Omega},\xi) p_\lambda(\xi)d\xi. \label{Holder-required-1}\end{eqnarray} 

Note that $\widetilde{V}_{G,\lambda}(\Omega)$ for $G\in \mathcal{C}$ given in \eqref{VG def} can be extended to continuous function $G: [0, \infty)\times \S^n\rightarrow [0, \infty)$ and $\Omega\in \mathcal{K}$.  It is clear that both of \eqref{uSSF} and \eqref{uSF2} are parabolic Monge-Amp\'ere type, their solutions exist for a short time. Therefore the flow \eqref{SSF}, as well as \eqref{SF2}, have short time solutions. 
 Let $\mathcal{M}_t=X(\mathbb{S}^n, t)$ be the smooth, closed and uniformly convex hypersurface parametrized by $X(\cdot, t)$, where $X(\cdot, t)$ is a smooth solution to the flow \eqref{SSF} with $t\in[0,T)$ for some constant $T>0$.  Let $\Omega_t$ be the convex body enclosed by $\mathcal{M}_t$ such that  $\Omega_t\in \mathcal{K}_0$ for all $t\in[0,T)$.  We now show that  $\widetilde{V}_{G,\lambda}(\cdot)$ remains unchanged along the flow \eqref{SSF}.

\begin{lem}\label{Volume}  Let $G\in \mathcal{G}_I^0$ and $\Psi\in \mathcal{G}_I^0$. Let $X(\cdot,t)$ be a smooth solution to the flow \eqref{SSF} with $t\in[0,T)$, and $\M_t=X(\S^n,t)$ be a smooth, closed and uniformly convex hypersurface.  Suppose that the origin lies in the interior of the convex body $\Omega_t$ enclosed by $\M_t$ for all $t\in[0,T)$. Then, for any $t\in[0,T)$, one has 
\begin{eqnarray}\label{Volume bound}
\widetilde{V}_{G,\lambda}(\Omega_t)=\widetilde{V}_{G,\lambda}(\Omega_0).
\end{eqnarray}
\end{lem}
\begin{proof} It follows from \eqref{Jac-1-1} that, by letting $x=\alpha_{\Omega}(\xi)$,  
$$\int_{\S^n}f{\psi}(u, x)\,dx=\int_{\S^n}f\frac{{\psi}(u, x)}{u}r^{n+1}{K} \,d\xi.$$
This, together with \eqref{eta def}, \eqref{rSSF} and \eqref{J(xi)},  yield  that 
\begin{eqnarray*}
\frac{d}{dt}\int_{\S^n}G(r,\xi)p_\lambda(\xi)\,d\xi
&=&\int_{\S^n}G_z(r,\xi)p_\lambda(\xi)r_t\,d\xi
\\&=&\int_{\S^n}G_z(r,\xi)p_\lambda(\xi)\left(-f\frac{{\psi}(u)}{u}r^{n+1}G^{-1}_z(r,\xi)p^{-1}_\lambda K+r\eta(t)\right)\,d\xi
\\&=&-\int_{\S^n}f\frac{{\psi}(u,x)}{u}r^{n+1}{K}d\xi+\eta(t) \int_{\S^n}rG_z(r,\xi) p_\lambda(\xi)d\xi\\&=&-\int_{\S^n}f\frac{{\psi}(u,x)}{u}r^{n+1}{K}d\xi+\int_{\S^n}f{\psi}(u,x)\,dx=0.
\end{eqnarray*}
In conclusion, $\widetilde{V}_{G,\lambda}(\cdot)$ remains unchanged along the flow \eqref{SSF}, and in particular, \eqref{Volume bound} holds for any $t\in[0,T)$.  
\end{proof}

The lemma below shows that the functional \begin{equation}\label{function}
\mathcal{J}(u)=\int_{\S^n}f{\Psi}(u,x)dx,
\end{equation} is monotone along the flow \eqref{SSF}.
\begin{lem}\label{mono} Let $G\in \mathcal{G}_I^0$ and $\Psi\in \mathcal{G}_I^0$. Let $X(\cdot,t)$, $\M_t$, and $\Omega_t$ be as in Lemma \ref{Volume}. Then the functional $\mathcal{J}$ defined in \eqref{function} is non-increasing along the flow \eqref{SSF}. That is, $\frac{\,d\mathcal{J}(u(\cdot, t))}{\,dt}\leq 0$, with equality if and only if $\mathcal{M}_t$ satisfies the elliptic equation \eqref{elliptic}.
\end{lem}
\begin{proof} Let $u(\cdot, t)$  be the support function of $\mathcal{M}_t$. 
It follows from \eqref{J(xi)}, \eqref{Holder-required-1}, and H\"{o}lder inequality that
 \begin{eqnarray*}  \int_{\mathbb{S}^n}f{\psi}(u,x)\!\,dx \!\! &=& \!\!\! \!\int_{\mathbb{S}^n} \!\! \Big(f^2{\psi}^2(u,x)r^n G_z^{-1}(r,\xi)p_\lambda^{-1}(\xi)u^{-1}K\Big)^{\frac{1}{2}} \Big(r^n G_z^{-1}(r,\xi)p_\lambda^{-1}(\xi)u^{-1}K\Big)^{-\frac{1}{2}}\!\!\,dx \\ \! \!&\le&\! \!\! \left(\int_{\mathbb{S}^n}f^2{\psi}^2(u,x)r^n G_z^{-1}(r,\xi)p_\lambda^{-1}(\xi)u^{-1}K\,dx\right)^{\frac{1}{2}} \left(\int_{\mathbb{S}^n} rG_z(r,\xi) p_\lambda(\xi)d\xi\right)^{\frac{1}{2}}.\end{eqnarray*}   Together with   \eqref{uSSF}, one gets  
\begin{eqnarray*}
\frac{\,d\mathcal{J}(u(\cdot, t))}{\,dt}=-\int_{\mathbb{S}^n}f^2{\psi}^2(u,x)r^n G_z^{-1}(r,\xi)p_\lambda^{-1}(\xi)u^{-1}Kdx+\frac{\left(\int_{\mathbb{S}^n}f {\psi} (u,x)dx\right)^2}{\int_{\mathbb{S}^n} rG_z(r,\xi)p_\lambda(\xi)d\xi}\leq 0. 
\end{eqnarray*}  Clearly, equality holds here if and only if equality holds for the  H\"{o}lder inequality, namely, there exists a constant $c(t)>0$ such that 
\begin{eqnarray}\label{relsult1}
f\frac{{\psi}(u,x)}{u}r^n G_z^{-1}(r,\xi)p_\lambda^{-1}(\xi)\big({\det}(\nabla^2 u+uI)\big)^{-1}=c(t).
\end{eqnarray} Moreover, it can be proved by \eqref{Holder-required-1} and \eqref{relsult1} that $c(t)=\eta(t)$ as follows: 
\[\eta(t)=\frac{\int_{\mathbb{S}^n}f{\psi}(u,x) dx}{\int_{\mathbb{S}^n} rG_z(r,\xi) p_\lambda(\xi)d\xi}=c(t) \frac{\int_{\mathbb{S}^n} r^{-n} G_z(r,\xi)p_\lambda(\xi) uK^{-1} dx}{\int_{\mathbb{S}^n}rG_z(r,\xi) p_\lambda(\xi)d\xi}=c(t).\] This concludes the proof. 
\end{proof}

Regarding the flow \eqref{SF2}, results similar to Lemmas \ref{Volume} and \ref{mono} can be obtained. For $\Psi\in \mathcal{G}_I^0$, let \begin{eqnarray*}
 \widehat{\Psi}_\varepsilon(s,x)=\int_0^s\frac{\widehat{\psi}_\varepsilon(t,x)}{t}\,dt\ \ \ \mathrm{for}\ \  (s,x)\in [0,\infty)\times{\S}^{n}.  \end{eqnarray*} It follows from \eqref{hatpsi eps} that  $\widehat{\Psi}_\varepsilon: [0,\infty)\times \S^n\rightarrow [0, \infty)$ is well defined and  continuous, such that, $\widehat{\Psi}_\varepsilon(0, x)=0$ for all $x\in\S^n.$ Moreover,  $\widehat{\Psi}_\varepsilon(s,x)$ is strictly increasing with respect to $s$ on $[0,\infty)$. Define the functional $\mathcal{J}_\varepsilon$ by
\begin{equation}\label{functional}
\mathcal{J}_\varepsilon(u)=\int_{\S^n}f\widehat{\Psi}_\varepsilon(u,x)dx.
\end{equation}

Following the same lines as the proofs for Lemmas  \ref{Volume} and \ref{mono}, one can prove that, along the flow \eqref{SF2}, $\mathcal{J}_\varepsilon$  is non-increasing  and $\widetilde{V}_{G,\lambda}(\cdot)$ remains unchanged.  In fact, it is clear that, for $\varepsilon>0$ small enough, the solutions to the flow \eqref{SF2} exist for a short time.  Let $\mathcal{M}_t^\varepsilon=X_{\varepsilon} (\mathbb{S}^n, t)$ be the smooth, closed and uniformly convex hypersurface parametrized by $X_{\varepsilon}(\cdot, t)$, where $X_{\varepsilon}(\cdot, t)$ is a smooth solution to the flow \eqref{SF2} with $t\in[0,T)$ for some constant $T>0$.  Let $\Omega^{\varepsilon}_t$ be the convex body enclosed by $\mathcal{M}^{\varepsilon}_t$ such that  $\Omega_t^{\varepsilon}\in \mathcal{K}_0$ for all $t\in[0,T)$. The support function of $\Omega_t^{\varepsilon}$ is $u_{\varepsilon}(x,t)$.  
\begin{lem}\label{monotone1}Let $G\in \mathcal{G}_I^0$ and $\Psi\in \mathcal{G}_I^0$. Then, the functional $\mathcal{J}_\varepsilon$ is non-increasing along the flow \eqref{SF2}. That is, $\frac{\,d\mathcal{J_{\varepsilon}}(u_{\varepsilon}(\cdot, t))}{\,dt}\leq 0$, with equality if and only if $\mathcal{M}^{\varepsilon}_t$ satisfies the elliptic equation   \eqref{elliptic eq}. Moreover,  
\begin{eqnarray*}
\widetilde{V}_{G,\lambda}(\Omega^{\varepsilon}_t)=\widetilde{V}_{G,\lambda}(\Omega_0),  \ \ \mathrm{for \ all} \,t\in[0,T).
\end{eqnarray*}
\end{lem}

\section{$C^0$ estimates} \label{section-3}
In this section, the uniform lower and upper bounds of the solutions to \eqref{SF2} are provided. Let ${\Omega_0}\in \mathcal{K}_0$ be the convex body enclosed by the initial hypersurface $\mathcal{M}_0$ of the flow \eqref{SF2}. As $\varepsilon\in (0, \delta)$ will be arbitrarily small, a constant $c_0>0$ can be found, such that, for all $\xi\in \S^n$, 
\begin{eqnarray}\label{ep}
r_{\Omega_0}(\xi)\ge c_0\ge10\varepsilon.
\end{eqnarray}  For simplicity, we shall omit the subscript/superscript $\varepsilon$ in  $X_\varepsilon(\cdot,t)$  (the solution to \eqref{SF2}), $\mathcal{M}_t^\varepsilon=X_\varepsilon(\S^n,t)$, $\Omega_t^{\varepsilon}$ (the convex body enclosed by $\mathcal{M}_t^\varepsilon$),  $u_\varepsilon(\cdot,t)$ and $r_{\varepsilon}(\cdot, t)$  (the support and radial functions of $\Omega_t^{\varepsilon}$) etc, if  no confusion occurs.

 \begin{lem}\label{upperbound}
Let $f$ and $p_\lambda$ be two smooth positive functions on $\mathbb{S}^n$, and $u_0$ be a positive and uniformly convex function. Let $u(\cdot,t)$ be a positive, smooth and uniformly convex solution to \eqref{uSF2}.
Let  $\Psi\in \mathcal{G}_I^0$ be a smooth function satisfying \eqref{Psi-comp-condition}.  
 Suppose that $G\in \mathcal{G}_I^0$ and $\psi=z\Psi_z$ satisfy \eqref{good-condition-2}. 
Then there is a constant $C>0$ depending only on $f$, $p_\lambda$, $G$, $\psi$ and $\Omega_0$, but independent of $\varepsilon$, such that
\begin{eqnarray}\label{upper bound2}
\max_{x\in\mathbb{S}^n}u(x, t)  \le  C \ \ \mathrm{and}\ \ \max_{x\in\mathbb{S}^n}|\nabla u (x, t)| \le C \ \ \ \mathrm{for\ all}\  t\in[0,T).
\end{eqnarray} 
\end{lem}
\begin{proof}   Combining with \eqref{functional} and Lemma \ref{monotone1},  there exists a constant $C_1>0$, independent of $\varepsilon$, such that, 
 \begin{eqnarray}\label{intbound1} C_1\geq 
\mathcal{J}_\varepsilon \big(u_0\big)\ge\mathcal{J}_\varepsilon\big(u(\cdot, t)\big)=\int_{{\S}^n}f(x)\widehat{\Psi}_\varepsilon(u,x)dx.
\end{eqnarray}
Let $x_t\in\S^n$ be such that  $u(x_t, t) =\max_{x\in \S^n} u(x, t)$. Without loss of generality, assume that $u(x_t, t)>10$ holds for some $t\in [0, T)$ (otherwise, there is nothing to prove). Let $$\Sigma_{\beta}=\{x\in \S^n: \langle x, x_t\rangle \ge \beta\}.$$ As $\varepsilon\in (0, \delta)$ with $\delta<1$, it can be checked from  Lemma \ref{pro} that, \begin{equation}\label{comp-max-supp}
u(x,t)\ge u(x_t,t)\langle x_t,x\rangle\ge\frac{1}{2}u(x_t,t)>5>2\varepsilon \ \ \ \mathrm{for\ all}\ x\in \Sigma_{1/2}.\end{equation}
It is easily checked that $\widehat{\Psi}_\varepsilon(s,x)\ge \Psi(s,x)-\Psi(2\varepsilon, x)\ge\Psi(s,x)-\Psi(5, x)$ for $(s,x)\in[2\varepsilon,\infty)\times\mathbb{S}^n$. Therefore, 
\begin{eqnarray*}
\int_{\S^n}f(x)\widehat{\Psi}_\varepsilon(u(x, t),x)\,dx&\ge&\int_{x\in\Sigma_{1/2}}f(x)\widehat{\Psi}_\varepsilon(u(x, t),x)\,dx
\\&\ge&\int_{x\in\Sigma_{1/2}} f(x) \widehat{\Psi}_\varepsilon \Big(\frac{u(x_t,t)}{2}, x\Big)\,dx
\\&\ge&  \min_{x\in \S^n}  \Big[\Psi\Big(\frac{u(x_t,t)}{2}, x\Big)-\Psi(5, x)\Big] \cdot \min_{x\in \S^n} f(x)\cdot \int_{\Sigma_{1/2}}\,dx.
\end{eqnarray*} Note that, the set $\Sigma_{1/2}$ may depend on $t$, but $ \int_{\Sigma_{1/2}}\,dx$ is a constant independent to $t$.  By \eqref{intbound1},  there exists $x_0\in \S^n$ (possibly depending on $t\in [0, T)$) such that
 \begin{eqnarray*}
\Psi\Big(\frac{u(x_t,t)}{2}, x_0\Big)-\Psi(5, x_0) =\min_{x\in \S^n}  \Big[\Psi\Big(\frac{u(x_t,t)}{2}, x\Big)-\Psi(5, x)\Big] \leq   C_1 \bigg( \min_{x\in \S^n} f(x)\cdot \int_{\Sigma_{1/2}}\,dx\bigg)^{-1}.
\end{eqnarray*} 
As $\Psi(t, x)$ is strictly increasing on $t\in (0, \infty)$ and $\lim_{t\rightarrow +\infty}\Psi(t, x)=+\infty$ for all $x\in \S^n$, $\Psi(t, \cdot)$ has its inverse function $\overline{\Psi}(t, \cdot)$  on $t\in [0, \infty)$. Moreover, $\overline{\Psi}: [0, \infty)\times \S^n$ is continuous such that $\lim_{t\rightarrow +\infty}\overline{\Psi}(t, x)=+\infty$ for all $x\in \S^n$.  Hence, 
 \[\max_{x\in \S^n} u(x,t) =u(x_t,t) \leq 2 \max_{x_0\in \S^n}\overline{\Psi}(\overline{C}_0, x_0)=: C,\] 
 where $\overline{C}_0=C_1 \bigg( \displaystyle\min_{x\in \S^n} f(x)\cdot \int_{\Sigma_{1/2}}\,dx\bigg)^{-1} +\max_{x\in\S^n}\Psi(5,x)$ and hence $C>0$ is a finite constant independent of $\varepsilon$ and $t$.  It follows from \eqref{nablau} and Lemma \ref{pro} that 
\[  \max_{x\in \S^n}|\nabla u(x, t)| \leq   \max_{x\in \S^n} r(x,t) =  \max_{x\in\S^n} u(x,t)\le C.\]
This completes the proof. 
\end{proof}

The following lemma provides uniform bound for $\eta_\varepsilon(t)$ defined in \eqref{eta epi def}.

\begin{lem}\label{etaepsbound}
Let $f$, $p_\lambda$, $G$ and $\Psi$ satisfy conditions stated in Lemma \ref{upperbound}. Let $u(\cdot,t)$ be a positive, smooth and uniformly convex solution to \eqref{uSF2}. 
Then  \begin{eqnarray}\label{etaepsilon bonud}
\frac{1}{C_2}\le\eta_\varepsilon(t)\le C_2\ \ \ \mathrm{for\ all}\  t\in[0,T),
\end{eqnarray} where $C_2>0$ is a constant  depending only on $f$, $p_\lambda$, $G$, $\Psi$ and $\Omega_0$, but independent of $\varepsilon$. \end{lem}
 
\begin{proof} 
 It follows from \eqref{ep},  Lemma \ref{monotone1}, $\,d\lambda(\xi)=p_{\lambda}(\xi)\,d\xi$, and  the fact that $G(t,\cdot)$ is strictly increasing on $t$  that 
\begin{eqnarray}\label{le1}\widetilde{V}_{G,\lambda}(\Omega_t) =\widetilde{V}_{G,\lambda}(\Omega_0)=\int_{\mathbb{S}^n}G\big(r_{\Omega_0}(\xi),\xi\big) \,d\lambda(\xi) \geq 
\int_{\mathbb{S}^n}G(c_0,\xi)p_\lambda(\xi)d\xi.
\end{eqnarray}This further implies $
\max_{x\in \mathbb{S}^n}r(x,t)\ge c_0 \ge 10\varepsilon.$  Moreover, by Lemmas \ref{pro} and \ref{upperbound} (in particular, \eqref{upper bound2}), one sees 
\begin{equation}  \sup_{(x, t)\in \S^n\times[0,T)}u(x, t)=\sup_{(x, t)\in \S^n\times[0,T)}r(x, t)\leq C. \label{max-s-r}\end{equation} 
Thus, for any $t\in [0, T)$, \begin{eqnarray}\label{upp-eta-top} \int_{\mathbb{S}^n}r(\xi, t)G_z(r(\xi, t),\xi) p_\lambda(\xi)d\xi \leq  \max_{(r,\xi)\in [0,C]\times{\S}^n}[rG_z(r,\xi)p_{\lambda}(\xi)] \cdot |\S^n|.\end{eqnarray}

Again let $x_t\in\mathbb{S}^n$ satisfy that $u({x_t},t)=\displaystyle\max_{x\in \S^n}u(x,t)$. Similar to \eqref{comp-max-supp},  for $\varepsilon\in (0, \delta)$ with $\delta<1$ and  for $x\in \Sigma_{1/2}$,  one has,
\begin{eqnarray} \label{comp-5-15} u(x,t)\ge\frac{1}{2}u({x_t},t)>\frac{c_0}{2} \ge5\varepsilon.\end{eqnarray} 
From \eqref{hatpsi eps}, the following fact holds: for  $x\in \Sigma_{1/2}$, 
\begin{eqnarray}\label{widetilde1}
\widehat{\psi}_\varepsilon\big(u(x,t),x\big)=\psi\big(u(x,t),x\big).
\end{eqnarray} By \eqref{max-s-r}, \eqref{comp-5-15} and  \eqref{widetilde1}, one gets  
 \begin{eqnarray*} \int_{\S^n}f\widehat{\psi}_\varepsilon(u(x,t),x) dx\geq \!\! 
   \int_{\Sigma_{1/2}}\!\!\!\!\! f\psi(u(x,t),x)\,dx    \geq \!   \min_{(s, x)\in [\frac{c_0}{2}, C]\times \mathbb{S}^n}\psi\big(s,x\big) \cdot \min_{x\in \S^n} f(x)\cdot \int_{\Sigma_{1/2}}\!\! \!\! \!\! \,dx. \end{eqnarray*} 
 This, together with \eqref{upp-eta-top}, further imply that, for all $t\in [0, T)$, 
 \begin{eqnarray} 
\eta_\varepsilon(t)&=&\frac{\int_{\S^n}f\widehat{\psi}_\varepsilon(u(x,t),x) dx}{\int_{\S^n}r(\xi,t)G_z(r(\xi,t),\xi) p_\lambda(\xi)d\xi} \nonumber \\ &\geq& \frac{\displaystyle\min_{(s, x)\in [\frac{c_0}{2}, C]\times \mathbb{S}^n}\psi\big(s,x\big)\cdot \displaystyle\min_{x\in \S^n} f(x)\cdot\int_{\Sigma_{1/2}}\!\! \!\! \!\! \,dx}{\displaystyle \max_{(t,\xi)\in [0,C]\times{\S}^n}[tG_z(t,\xi)p_{\lambda}(\xi)] \cdot |\S^n|}>0.  \label{first-ineq}
\end{eqnarray}  
It follows from Lemma \ref{monotone1} and inequality \eqref{le1}that  
\begin{eqnarray*}\nonumber \int_{\{\xi\in\mathbb{S}^n: r(\xi, t)\geq \frac{c_0}{2} \}}G\big(r(\xi, t),\xi\big)p_\lambda(\xi)d\xi&=&\widetilde{V}_{G,\lambda}(\Omega_t)-\int_{\{\xi\in\mathbb{S}^n: r(\xi, t)< \frac{c_0}{2} \}}G\big(r(\xi, t),\xi\big)p_\lambda(\xi)d\xi \\ &\geq&
  \int_{\mathbb{S}^n}G(c_0,\xi)p_\lambda(\xi)d\xi-\int_{\mathbb{S}^n}G\Big( \frac{c_0}{2},\xi\Big)p_\lambda(\xi)d\xi \\&\geq& \min_{\xi\in \mathbb{S}^n} \Big[ G(c_0,\xi)p_{\lambda}(\xi)-G\Big( \frac{c_0}{2},\xi\Big)p_{\lambda}(\xi) \Big] \cdot |\mathbb{S}^n|. \end{eqnarray*}
This further implies
\begin{eqnarray}\nonumber
\int_{\mathbb{S}^n}r(\xi,t)G_z(r(\xi,t),\xi) p_\lambda(\xi)d\xi  &\geq& \int_{\{\xi\in\mathbb{S}^n: r(\xi, t) \geq \frac{c_0}{2} \}} r(\xi,t)G_z(r(\xi,t),\xi) p_\lambda(\xi)d\xi  \\ &\geq& a_0  \int_{\{\xi\in\mathbb{S}^n: r(\xi, t)\geq \frac{c_0}{2} \}} G(r(\xi,t),\xi) p_\lambda(\xi)d\xi \nonumber \\ &\geq & a_0 \min_{\xi\in \mathbb{S}^n} \Big[ G(c_0,\xi)p_{\lambda}(\xi)-G\Big( \frac{c_0}{2},\xi\Big)p_{\lambda}(\xi) \Big] \cdot |\mathbb{S}^n|, \label{3.11}
\end{eqnarray} where the constant $a_0>0$ can be taken as $$a_0=\min_{(t, \xi)\in [\frac{c_0}{2}, C]\times \mathbb{S}^n}\frac{tG_z(t, \xi)}{G(t, \xi)}>0.$$

Recall that $\delta\in (0, 1)$  is a constant satisfying \eqref{def-delta}, namely,   \begin{eqnarray*} \max_{(s,\xi)\in [0,\delta]\times\S^n} sG_z(s,\xi)\le 1. \end{eqnarray*} 
It follows from  \eqref{hatpsi eps} and \eqref{constant-c-0} that  $\widehat{\psi}_\varepsilon(s,x)\le C_0=\max\{1, \max_{[0, 2]\times \S^n}\psi(s, x)\}$ on $[0,2\varepsilon]$, 
and $\widehat{\psi}_\varepsilon(s,x)=\psi(s,x)$ on $[2\varepsilon,\infty)$.
Hence, \eqref{max-s-r} implies that, for all $t\in [0, T)$, 
\begin{eqnarray*}  
\int_{\S^n}f\widehat{\psi}_\varepsilon(u(x, t), x)\,dx&=& \int_{\{x\in\S^n: u(x, t)\in[0,2\varepsilon)\cup[2\varepsilon,C]\}} f\widehat{\psi}_\varepsilon(u(x, t), x)\,dx 
\\&\le&\int_{\S^n} \Big(C_0+\psi\big(u(x,t), x\big)\Big)f\,dx \\ &\le& \Big\{C_0+\max_{(s,x)\in [0,C]\times\S^n} \psi\big(s, x\big)\Big\}\cdot\max_{x\in\S^n} f(x)\cdot |\S^n|. 
\end{eqnarray*} Together with \eqref{3.11}, one has, for all $t\in [0, T)$,
 \begin{eqnarray} 
\eta_\varepsilon(t)&=&\frac{\int_{\S^n}f\widehat{\psi}_\varepsilon(u(x,t),x) dx}{\int_{\S^n}r(\xi,t)G_z(r(\xi,t),\xi) p_\lambda(\xi)d\xi} \nonumber \\ &\leq& \frac{\Big[C_0+\displaystyle\max_{(s,x)\in [0,C]\times\S^n} \psi\big(s, x\big)\Big]\cdot\displaystyle\max_{x\in\S^n} f(x)\cdot|\S^n| 
}{a_0 \min_{\xi\in \mathbb{S}^n} \Big[ p_{\lambda}(\xi)G(c_0,\xi)-p_{\lambda}(\xi)G\Big( \frac{c_0}{2},\xi\Big)\Big]}. \label{eta-est-2108}
\end{eqnarray} In view of \eqref{first-ineq} and \eqref{eta-est-2108}, a constant $C_2>0$ independent of $\varepsilon$ and $t\in [0, T)$  can be found so that \eqref{etaepsilon bonud} holds.   
\end{proof}
  
 The following lemma provides a uniform lower bound for $u(x, t)$ on $t\in [0, T)$. 
  \begin{lem}\label{lowerbound} Let $f$, $p_\lambda$, $G$ and $\Psi$ satisfy conditions stated in Lemma \ref{upperbound}.  Let $u(\cdot,t)$ be a positive, smooth and uniformly convex solution to \eqref{uSF2}.  
Then there is a constant $C_{\varepsilon}>0$ depending only on $\varepsilon$, $f$, $p_\lambda$, $G$, $\psi$ and $\Omega_0$, such that, for all $t\in[0,T)$,
\begin{eqnarray*}
\min_{x\in \mathbb{S}^n}u(x, t)\ge1/C_\varepsilon.
\end{eqnarray*}
\end{lem}

\begin{proof} Let $\bar{x}_t\in \S^n$ be such that $u(\bar{x}_t,t)= \min_{x\in \mathbb{S}^n}u(x,t)$. Without loss of generality, assume that  $u(\bar{x}_t,t) <\varepsilon$ holds for some $t\in [0, T)$ (as otherwise, there is nothing to prove).   
From \eqref{nablau}  and Lemma \ref{pro}, $\alpha^*(\bar{x}_t)=\bar{x}_t$, $\nabla u(\bar{x}_t,t)=0$, and
\begin{equation}\label{min-equal-1} r(\bar{x}_t, t)=\min_{\xi\in \S^n} r(\xi, t)= \min_{x\in \mathbb{S}^n}u(x,t)=u(\bar{x}_t,t).\end{equation}   It follows from \eqref{hatpsi eps} and  the fact that $\nabla^2 u(\bar{x}_t,t)$ is positive semi-definite that 
\begin{eqnarray} \widehat{\psi}_\varepsilon(u(\bar{x}_t,t),\bar{x}_t)G_z(u(\bar{x}_t,t),\bar{x}_t)^{-1}=\left(u(\bar{x}_t,t)\right)^{1+\varepsilon},\nonumber \\  
 \det(\nabla^2 u(\bar{x}_t,t)+u(\bar{x}_t,t)I)\geq u(\bar{x}_t,t)^n. \label{estimate-curvature}\end{eqnarray} Together with \eqref{curvature-formula-1} and \eqref{uSF2}, one has, 
\begin{eqnarray*}
\partial_tu(\bar{x}_t,t)&\ge&-f(\bar{x}_t) \left(u(\bar{x}_t,t)\right)^{1+\varepsilon}\left(u(\bar{x}_t,t)\right)^n p_\lambda^{-1}(\bar{x}_t)\left(u(\bar{x}_t,t)\right)^{-n} +u(\bar{x}_t,t)\eta_\varepsilon(t)
\\&\ge&f(\bar{x}_t)u(\bar{x}_t,t)\left(-p_\lambda^{-1}(\bar{x}_t)\left(u(\bar{x}_t,t)\right)^\varepsilon+\frac{\eta_\varepsilon(t)}{f(\bar{x}_t)}\right).
\end{eqnarray*}
This implies either $\partial_tu(\bar{x}_t,t)\ge0$, or
\[u(\bar{x}_t,t)>\left(\frac{\eta_\varepsilon(t)\displaystyle\min_{x\in \S^n}p_\lambda(\xi)}{\displaystyle\max_{x\in \S^n}f(x)}\right)^\frac{1}{\varepsilon}.\]
In view of \eqref{ep} and \eqref{etaepsilon bonud}, one gets 
\begin{eqnarray*}
\min_{x\in \mathbb{S}^n}u(x, t)\ge\min\bigg\{\varepsilon,  \left(\frac{  \min_{\xi\in\S^n}p_\lambda(\xi)}{C_2 \cdot \max_{x\in \mathbb{S}^n}f(x)}\right)^\frac{1}{\varepsilon} \bigg\}=: \frac{1}{C_{\varepsilon}},
\end{eqnarray*} where $C_{\varepsilon}>0$ is a constant independent of $t\in [0, T)$. This completes the proof.
\end{proof}

Note that the uniform lower bound for  $u(\cdot,t)$ in Lemma \ref{lowerbound} depends on $\varepsilon\in (0,\delta)$. However, such a lower bound may go to zero as $\varepsilon\to 0^+$. The following lemma gives bounds (both upper and lower), which are independent of $\varepsilon$, for the maximal and minimal widths of $\Omega_t$ on $t\in [0, T)$.  Recall that the maximal and  minimal widths of $\Omega$ defined in \eqref{mmwide}, respectively, are
\begin{eqnarray*} 
w_\Omega^+=\max_{x\in\mathbb{S}^n}\{u_\Omega(x)+u_\Omega(-x)\}\ \ \mathrm{and}\ \ 
w_\Omega^-=\min_{x\in\mathbb{S}^n}\{u_\Omega(x)+u_\Omega(-x)\}.
\end{eqnarray*}

\begin{lem}\label{widthbound}Let $f$, $p_\lambda$, $G$ and $\Psi$ satisfy conditions stated in Lemma \ref{upperbound}.   Let $u(\cdot,t)$ be a positive, smooth and uniformly convex solution to \eqref{uSF2}.  
Then there is a constant $C_3>0$ depending only on $f$, $p_\lambda$, $G$, $\Psi$, and $\Omega_0$, but independent of $\varepsilon$, such that,  for all $t\in[0,T)$, 
\begin{eqnarray*}
1/C_3 \le w_{\Omega_t}^-\le w_{\Omega_t}^+\le C_3.
\end{eqnarray*}
\end{lem}
\begin{proof} For each $t\in [0, T)$, let $l_t=2w_{\Omega_t}^-$ and $\mathbb{L}_t=\mathbb{B}_R^n\times[-l_t,l_t]$, where $\mathbb{B}_R^n$ is the n-dimension ball centered at the origin with radius $R=2\sqrt{n}C$ for $C$ the constant given in \eqref{upper bound2} by Lemma \ref{upperbound}. Let $\{e_1,\cdots,e_{n+1}\}$ be a local orthonormal frame field on $\mathbb{R}^{n+1}$ and (without loss of generality) $u_{\Omega_t}(e_{n+1})\leq \omega_{\Omega_t}^{-}$. Then, the support function $u_{\mathbb{L}_t}(e_{n+1})=l_t$. Moreover, $\Omega_t\subseteq\mathbb{L}_t$ and hence  $r_{\mathbb{L}_t}(\xi)\le\sqrt{2}R$ for all $\xi\in\mathbb{S}^n$.  Due to the fact that $G(t, \cdot)$ is increasing on $t\in [0, \infty)$, one gets 
\begin{equation} \label{VGL} 
\widetilde{V}_{G,\lambda}(\Omega_t)\le\widetilde{V}_{G,\lambda}(\mathbb{L}_t)\!=\!\! \int_{\mathbb{S}^n}\!\!G\big(r_{\mathbb{L}_t}(\xi),\xi\big) p_{\lambda}(\xi) \,d\xi \le A\cdot |\mathbb{S}^n|,  
\end{equation} where $A=\max_{\xi\in\mathbb{S}^n}\big[G(\sqrt{2}R,\xi)p_\lambda(\xi)\big]$ is a positive constant. 

For any $\beta\in (0, 1)$, let $\widehat{\Sigma}_{\beta}=\{\xi\in \S^n: |\langle \xi, e_{n+1}\rangle|\leq \beta  \}.$ It is easily checked that \begin{equation}\label{est-beta-area}\int_{\widehat{\Sigma}_{\beta}}\,d\xi \leq \beta \cdot \alpha_n,\end{equation}  holds for a fixed constant  $\alpha_n>0$  only depending on $n$. In particular, $|\S^n|\leq \alpha_n$. Together with \eqref{VGL}, one can get  $\widetilde{V}_{G,\lambda}(\Omega_0)=\widetilde{V}_{G,\lambda}(\Omega_t)\leq \alpha_n A$ by Lemma \ref{monotone1}.  Let    
\begin{eqnarray}\label{delta}
\delta_1= \frac{\widetilde{V}_{G,\lambda}(\Omega_t)}{2\alpha_nA} = \frac{\widetilde{V}_{G,\lambda}(\Omega_0)}{2\alpha_nA}\in (0, 1/2].
\end{eqnarray}   It is easily checked that, for all $\xi \in \S^n\setminus \widehat{\Sigma}_{\delta_1}$,  
\begin{eqnarray}\label{rL}
r_{\mathbb{L}_t}(\xi)\leq \frac{u_{\mathbb{L}_t}(e_{n+1})}{|\langle \xi, e_{n+1}\rangle|} \le\frac{l_t}{\delta_1}.
\end{eqnarray}
By \eqref{est-beta-area}, \eqref{delta} and \eqref{rL}, one has
\begin{eqnarray}\label{VG ineq}
\nonumber\widetilde{V}_{G,\lambda}(\mathbb{L}_t)&=&\int_{ \S^n\setminus \widehat{\Sigma}_{\delta_1}}G(r_{\mathbb{L}_t}(\xi),\xi)p_\lambda(\xi)d\xi
+\int_{ \widehat{\Sigma}_{\delta_1}}G(r_{\mathbb{L}_t}(\xi),\xi)p_\lambda(\xi)d\xi
\nonumber\\&\le&\int_{ \S^n\setminus \widehat{\Sigma}_{\delta_1}}G(l_t/\delta_1,\xi)p_\lambda(\xi)d\xi
+\int_{ \widehat{\Sigma}_{\delta_1}}G(\sqrt{2}R,\xi)p_\lambda(\xi)d\xi
\nonumber\\&\le&\int_{\mathbb{S}^n}G(l_t/\delta_1,\xi)p_\lambda(\xi)d\xi +A\alpha_n\delta_1
\nonumber\\&\le&\int_{\mathbb{S}^n}G(l_t/\delta_1,\xi)p_\lambda(\xi)d\xi +\frac{1}{2}\widetilde{V}_{G,\lambda}(\Omega_t).\nonumber
\end{eqnarray}
Together with \eqref{VGL} and Lemma \ref{monotone1}, the following holds:  \begin{eqnarray*}
\frac{\widetilde{V}_{G,\lambda}(\Omega_t)}{2} \le\int_{\mathbb{S}^n}G(l_t/\delta_1,\xi)p_\lambda(\xi)d\xi \leq  |\S^n| \cdot  \max_{\xi\in \S^n}p_\lambda(\xi) \cdot  \max_{\xi\in \S^n} G(l_t/\delta_1,\xi).
\end{eqnarray*} Hence, there exists $\xi_0\in \S^n$ (possibly depending on $t$) such that 
 \begin{eqnarray}\label{VG1}
G(l_t/\delta_1,\xi_0)\geq \frac{\widetilde{V}_{G,\lambda}(\Omega_t)}{2   |\S^n| \cdot \max_{\xi\in \S^n}p_\lambda(\xi)}.
\end{eqnarray} 
As $G_z>0$, then $G(t, \cdot)$ is strictly increasing on $t\in (0, \infty)$. Thus,  for any $\xi\in \S^n$, the function $G(t, \xi): t\mapsto [0, \infty)$ has its inverse function, which will be denoted by $\overline{G}(t, \xi): t\mapsto [0, \infty).$  It follows from \eqref{VG1} that  
\begin{eqnarray*}  
  l_t  \geq  \delta_1  \cdot \overline{G}\bigg(\frac{\widetilde{V}_{G,\lambda}(\Omega_t)}{2   |\S^n| \cdot \max_{\xi\in \S^n}p_\lambda(\xi)}, \xi_0\bigg).
\end{eqnarray*} Note that the function $\overline{G}(t, \xi): [0, \infty)\rightarrow [0, \infty)$ is continuous. This, together with $l_t=2w_{\Omega_t}^-$, further imply that, for all $t\in [0, T)$,  
\begin{eqnarray*}
  w_{\Omega_t}^- =\frac{l_t}{2}  \geq  \frac{\delta_1}{2} \cdot \min_{\widetilde{\xi}\in \S^n} \overline{G}\bigg(\frac{\widetilde{V}_{G,\lambda}(\Omega_0)}{2   |\S^n| \cdot \max_{\xi\in \S^n}p_\lambda(\xi)},\widetilde{\xi}\bigg),
\end{eqnarray*}  and the constant in the right hand is clearly independent of $\varepsilon$ and $t\in [0, T)$. 

On the other hand, by Lemma \ref{upperbound}, one sees that $$w_\Omega^+=\max_{x\in\mathbb{S}^n}\{u_\Omega(x)+u_\Omega(-x)\}\leq 2C.$$ So, the desired constant $C_3$ can be given by: 
$$\frac{1}{C_3}=\min \left\{ \frac{1}{2C}, \ \frac{\delta_1}{2} \cdot \min_{\widetilde{\xi}\in \S^n} \overline{G}\bigg(\frac{\widetilde{V}_{G,\lambda}(\Omega_0)}{2   |\S^n| \cdot \max_{\xi\in \S^n}p_\lambda(\xi)},\widetilde{\xi}\bigg)\right\}.$$ This concludes the proof. 
\end{proof}

The following lemma gives the $C^0$ estimates for the flow \eqref{SSF}.

\begin{lem}  Let $f$ and $p_\lambda$ be two smooth positive functions on $\mathbb{S}^n$. Let $\Psi\in \mathcal{G}_I^0$ be a  smooth function such that \eqref{Psi-comp-condition}. 
  Moreover,  let  $G\in \mathcal{G}_I^0$   be a smooth function such that   \eqref{good-condition-1}.     Let $u(\cdot,t)$ for $t\in[0,T)$ be a positive, smooth and uniformly convex solution to \eqref{uSSF}.
Then there is a constant $C>0$ depending only on $f$, $p_\lambda$, $G$, $\Psi$, and $\Omega_0$, such that, for all $t\in [0, T)$, 
\begin{eqnarray*} 
1/C\le \min_{x\in \S^n} u(x,t)\le\max_{x\in \S^n} u(x,t)\le C \ \ \mathrm{and} \ \  
\max_{\mathbb{S}^n}|\nabla u(x, t)|\le C.
\end{eqnarray*}
\end{lem}
\begin{proof} Consider the flow defined in \eqref{SSF} or \eqref{uSSF}.   Lemma \ref{mono} asserts that $\frac{d}{dt}\mathcal{J}(u(\cdot,t))=0$ holds if and only if $u(\cdot,t)$ solves \eqref{elliptic}. Following the proofs in   Lemmas \ref{upperbound} and \ref{etaepsbound},   the following statements hold for all $t\in [0, T)$: 
\begin{eqnarray}
\label{upper11}\max_{x\in \mathbb{S}^n}u(x,t)\le C,\\
\label{gradient11}\max_{x\in \mathbb{S}^n}|\nabla u|(x,t)\le C,\\
1/C\le\eta(t)\le C,\label{etabound1}
\end{eqnarray}
where the constant $C>0$ (abuse of notations) depends only on $f, p_\lambda, G, \Psi$ and  $\Omega_0$ but is independent of $t\in [0, T)$.

Let $\bar{x}_t\in\S^n$ satisfy that $u(\bar{x}_t,t)= \min_{x\in \S^n}u(x,t).$ Clearly, $\nabla u(\bar{x}_t,t)=0$ and $\nabla^2 u(\bar{x}_t,t)\ge 0$. A calculation similar to that in Lemma \ref{lowerbound}, by using \eqref{min-equal-1} and  \eqref{estimate-curvature}, shows that \begin{eqnarray}
\nonumber\partial_t u(\bar{x}_t,t) &\geq&-f(\bar{x}_t)\psi(u(\bar{x}_t,t), \bar{x}_t)G_z(u(\bar{x}_t,t), \bar{x}_t)^{-1}p_\lambda^{-1}(\bar{x}_t) +u(\bar{x}_t,t)\eta(t) \\&=&\frac{\psi(u(\bar{x}_t,t), \bar{x}_t)}{G_z(u(\bar{x}_t,t), \bar{x}_t)}\eta(t)\left(\frac{G_z(u(\bar{x}_t,t), \bar{x}_t)u(\bar{x}_t,t)}{\psi(u(\bar{x}_t,t), \bar{x}_t)}-\frac{f(\bar{x}_t)}{p_\lambda(\bar{x}_t)\eta(t)}\right).\label{p11}
\end{eqnarray} Recall that $G$ and $\psi$ satisfy \eqref{good-condition-1}, i.e.,  
$\displaystyle\liminf_{s\to0^+}\frac{sG_z(s,x)}{\psi(s,x)}=\infty$ holds for all $x \in \S^n$.  Let $M$ be a finite constant such that $$M>C \cdot \max_{x\in \S^n}\Big(\frac{f(x)}{p_{\lambda}(x)}\Big).$$ Then, there exists $\delta_0$ depending on $M$ so that for any $s\in (0, \delta_0)$,  
\begin{eqnarray}\label{part i}
\frac{sG_z(s,\bar{x}_t)}{\psi(s, \bar{x}_t)}\ge M.
\end{eqnarray} Without loss of generality, let $u(\bar{x}_t,t)<\delta_0$.  By \eqref{etabound1},  \eqref{p11} and  \eqref{part i}, one has 
\begin{eqnarray*}
 \partial_t u(\bar{x}_t,t) \geq \frac{\psi(u(\bar{x}_t,t), \bar{x}_t)}{G_z(u(\bar{x}_t,t), \bar{x}_t)}\eta(t)\left(M-\frac{f(\bar{x}_t)}{p_\lambda(\bar{x}_t)\eta(t)}\right)>0.
\end{eqnarray*} This further yields, for all $t\in [0, T)$,  \begin{eqnarray}
u(\bar{x}_t,t)=\min_{x\in \mathbb{S}^n}u(x,t)\ge\min\big\{\delta_0, c_0\big\}=: 1/C, \label{lower11}
\end{eqnarray} where $c_0$ is the constant given in \eqref{ep}, and $C$ (abuse of notations) is a positive constant  depending only on $f$, $p_\lambda$, $G$, $\Psi$, and $\Omega_0$ but independent of  $t\in [0, T)$.  This concludes the proof. \end{proof}

\section{Solutions to the extended Musielak-Orlicz-Gauss problem} \label{section-4} 
This section is devoted to the existence of solutions to the extended Musielak-Orlicz-Gauss problem. 
We first prove Theorem \ref{main3} in Section \ref{sec-main-3-1}, 
which provides a smooth solution to the  extended Musielak-Orlicz-Gauss problem. 
By the technique of approximation, Theorem \ref{main1} is proved in Section \ref{sec-main-1-1}, 
and this gives a solution to extended Musielak-Orlicz-Gauss problem for general data. 

\subsection{Properties of $\widetilde{V}_{G, \lambda}(\cdot)$ and $\widetilde{C}_{G,\lambda}(\cdot,\cdot)$  on $\mathcal{K}$} 
{~}
 
 Properties for $\widetilde{V}_{G, \lambda}(\cdot)$ and $\widetilde{C}_{G,\lambda}(\cdot,\cdot)$ on $\mathcal{K}_0$ when $G\in \mathcal{C}$ have been discussed in \cite{HSYZ-21}. We now discuss the case for $G: [0, \infty)\times \S^n\rightarrow [0, \infty)$ and $\Omega\in \mathcal{K}$. Recall that, for $\Omega\in \mathcal{K}$  a convex body containing the origin $o$, $N(\Omega, o)$ is  the normal cone of $\Omega$ at $o$,  namely, $$N(\Omega, o)=\big\{y \in \mathbb{R}^{n+1}: \langle \tilde{y}, y\rangle\leq 0 \ \ \mathrm{for}\ \tilde{y} \in \Omega\big\}.$$
Let $N^*(\Omega, o)$ be the dual cone of $N(\Omega, o)$ which takes the following form: 
$$N^*(\Omega, o)=\big\{\tilde{y} \in \mathbb{R}^{n+1}: \langle \tilde{y}, y\rangle\leq 0\ \ \mathrm{for} \ y\in N(\Omega, o)\big\}.$$ It can be checked that $N^*(\Omega, o)$ is the closure of the set $\{ay: y\in \Omega \ \ \mathrm{and}\ a\geq 0\}$.
 It is well-known that \begin{equation} \label{radius=0} \left\{\begin{array}{ll} r_{\Omega}(\xi)=0 & \mathrm{if}\ \ \xi\in \S^n\setminus N^*(\Omega, o),\\ r_{\Omega}(\xi) \geq 0 & \mathrm{if}\ \ \xi\in \S^n \cap N^*(\Omega, o). \end{array}\right. \end{equation}  Moreover, for each $\omega\subseteq\S^n$, up to sets of $\,d\xi$-measure $0$,  
\begin{equation}\label{dual-2-3} \alpha_{\Omega}^*(\omega\setminus N(\Omega, o))=\alpha^*_{\Omega}(\omega)\cap N^*(\Omega, o).\end{equation}  In particular, as $(\S^n\cap N(\Omega, o))\setminus N(\Omega, o)$ is empty,  up to sets of $\,d\xi$-measure $0$, \begin{equation}  \alpha^*_{\Omega}\big(\S^n\cap N(\Omega, o)\big)\subseteq \S^n\setminus N^*(\Omega, o).\label{dual-cone-relation-1} \end{equation}

Let $\lambda$ be a nonzero finite Borel measure on $\S^n$ that is absolutely continuous with respect to $\,d\xi$. Let $G:  [0, \infty)\times \S^n\rightarrow [0, \infty)$  be a continuous function such that $G(0, \xi)=0$ for all $\xi\in \S^n$. Let  $\Omega$ be a nonempty convex compact set in $\mathbb{R}^{n+1}$ containing the origin. We still define $\widetilde{V}_{G, \lambda}(\cdot)$  as the one in \eqref{VG def}, that is, 
\begin{eqnarray*}
\widetilde{V}_{G,\lambda}(\Omega)=\int_{\mathbb{S}^n}G(r_\Omega(\xi),\xi)d\lambda(\xi).
\end{eqnarray*}  If $\Omega$ is a convex compact set  containing the origin whose interior is empty, then $\Omega$ is contained in a lower-dimensional subspace of $\mathbb{R}^{n+1}$. Hence, the $\,d\xi$-measure of the set $\{\xi\in \S^n: r_{\Omega}(\xi)>0\}$  must be $0$ and  its $\lambda$-measure is $0$ as well. In this case,   \begin{eqnarray*}
\widetilde{V}_{G,\lambda}(\Omega)=\int_{\{\xi\in \S^n: r_{\Omega}(\xi)>0\}}G(r_\Omega(\xi),\xi)d\lambda(\xi)+\int_{\{\xi\in \S^n: r_{\Omega}(\xi)=0\}}G(r_\Omega(\xi),\xi)d\lambda(\xi)=0.
\end{eqnarray*} Equivalently, if $\widetilde{V}_{G,\lambda}(\Omega)>0$, then the convex compact set $\Omega$ in $\mathbb{R}^{n+1}$ containing the origin must have nonempty interior and hence $\Omega\in \mathcal{K}$.

The following result regarding $\widetilde{V}_{G, \lambda}(\cdot)$ shall be used.  Its proof follows from an argument similar to these in \cite{BF19,GHXY18}. For completeness, a short proof is included here. 
\begin{proposition}\label{Prop-V-G-1}  Let $G:  [0, \infty)\times \S^n\rightarrow [0, \infty)$  be a continuous function such that $G(0, \xi)=0$ for all $\xi\in \S^n$. Let $\lambda$ be a nonzero finite Borel measure on $\S^n$ that is absolutely continuous with respect to $\,d\xi$.   Let $\Omega_i\in \mathcal{K}$  for each $i\in \mathbb{N}$ and let $\Omega$ be a compact convex set in $\mathbb{R}^{n+1}$ containing the origin.  If $\Omega_i$ converges to $\Omega$ in the Hausdorff metric, then \begin{equation}\label{lim-V-G-5} \lim_{i\rightarrow \infty}  \widetilde{V}_{G, \lambda}(\Omega_i)=\widetilde{V}_{G, \lambda}(\Omega).\end{equation}
\end{proposition}

\begin{proof} First of all, if $\Omega\in \mathcal{K}_0$, then we can assume (without loss of generality) that $\Omega_i\in \mathcal{K}_0$ for each $i\in \mathbb{N}$. In this case, $\Omega_i\rightarrow \Omega$ in the Hausdorff metric is equivalent to $r_{\Omega_i}\rightarrow r_{\Omega}$ uniformly on $\S^n$. Hence, \eqref{lim-V-G-5} is an immediate consequence of the dominated convergence theorem. 

It has been proved in \cite[Lemma 2.2]{BF19} and \cite[Lemma 3.2]{GHXY18} that \begin{eqnarray}\label{limit-0-6-26} \lim_{i\rightarrow \infty} r_{\Omega_i}(\xi)=r_{\Omega}(\xi)=0\end{eqnarray} for $\,d\xi$-almost all $\xi\in \mathbb{S}^n$ if $\Omega$ is a compact convex set containing the origin and its interior is empty; while if $\Omega\in \mathcal{K}$ with $o\in \partial \Omega$,
\begin{eqnarray}\label{limit-1-6-26}\lim_{i\rightarrow \infty} r_{\Omega_i}(\xi)=r_{\Omega}(\xi) \ \ \mathrm{for} \ \ \xi\in \S^n\setminus \partial N^*(\Omega, o).\end{eqnarray}  Recall that  $\,d\xi$-measure  of the set $\S^n \cap\partial N^*(\Omega, o)$ is $0$. Hence both \eqref{limit-0-6-26} and \eqref{limit-1-6-26} hold for $\lambda$-almost all $\xi$. By the dominated convergence theorem, one  immediately gets that, if $\Omega\in \mathcal{K}$ with $o\in \partial \Omega$,  \begin{eqnarray*} 
 \lim_{i\rightarrow \infty}  \widetilde{V}_{G, \lambda}(\Omega_i)&=&\lim_{i\rightarrow \infty} \int_{\S^n\setminus \partial N^*(\Omega, o)}G(r_{\Omega_i}(\xi), \xi)\,d\lambda(\xi)\\&=& \int_{\S^n\setminus \partial N^*(\Omega, o)}\lim_{i\rightarrow \infty} G(r_{\Omega_i}(\xi), \xi)\,d\lambda(\xi)\\&=& \int_{\S^n\setminus \partial N^*(\Omega, o)} G(r_{\Omega}(\xi), \xi)\,d\lambda(\xi)=\widetilde{V}_{G, \lambda}(\Omega),
\end{eqnarray*}     and if $\Omega$ is a compact convex set containing the origin and its interior is empty, 
 \begin{eqnarray*} 
 \lim_{i\rightarrow \infty}  \widetilde{V}_{G, \lambda}(\Omega_i)=  \int_{\S^n}\lim_{i\rightarrow \infty} G(r_{\Omega_i}(\xi), \xi)\,d\lambda(\xi) =  \int_{\S^n} G(0, \xi)\,d\lambda(\xi)=0=\widetilde{V}_{G, \lambda}(\Omega).
\end{eqnarray*} This completes the proof. \end{proof}

 Recall Definition \ref{def-unified}  for the measure $\widetilde{C}_{\Theta}(\Omega,\cdot)$: $$
\widetilde{C}_{\Theta}(\Omega,\omega)=  \int_{{\mathrm{\alpha}_{\Omega}^*}(\omega\setminus N(\Omega, o))} \frac{r_\Omega(\xi)G_z(r_\Omega(\xi),\xi)}{\psi(u_\Omega(\alpha_\Omega(\xi)),\alpha_\Omega(\xi))}\,d\lambda(\xi),$$
 where $\Omega\in \mathcal{K}$,   $\psi=z\Psi_z$, and $\Theta=(G, \Psi, \lambda)$ is a triple containing two Musielak-Orlicz functions $G, \Psi$ defined on $[0, \infty)\times \S^n$  and  a measure $\lambda$ defined on $\S^n$.   It is  easily checked that $\widetilde{C}_{\Theta}(\Omega,\omega)=0$ when $\omega\subseteq \S^n\cap N(\Omega, o)$ (in this case, ${\mathrm{\alpha}_{\Omega}^*}(\omega\setminus N(\Omega, o))$ is an empty set). The most interesting case is when $\psi(t, \xi)=1$ for all $(t, \xi)\in [0, \infty)\times \S^n$, namely, the measure $\widetilde{C}_{G,\lambda}(\Omega,\cdot)$ defined in \eqref{GDOCM def-log}: for each Borel set $\omega\subseteq\S^n$,  \begin{eqnarray*}
\widetilde{C}_{G,\lambda}(\Omega,\omega) = \int_{{\mathrm{\alpha}_{\Omega}^*}(\omega\setminus N(\Omega, o))} r_\Omega(\xi)G_z(r_\Omega(\xi),\xi)\,d\lambda(\xi). 
\end{eqnarray*} Indeed, if $G\in \mathcal{G}_I^0$ and $\Omega\in \mathcal{K}$, then $\widetilde{C}_{G,\lambda}(\Omega,\omega)$ in \eqref{GDOCM def-log} does define a finite measure on $\S^n$. First of all, the integral \eqref{GDOCM def-log} does exists, and is nonnegative and finite for each $\Omega\in \mathcal{K}$. Moreover, if $\omega$ is an empty set, then $\widetilde{C}_{G,\lambda}(\Omega,\omega)=0$. The countable additivity is an immediate consequence of the following facts (see e.g., the proof of \cite[Proposition 5.2]{GHXY18}): for disjoint Borel sets $\omega_i\subseteq \S^n$, $i\in \mathbb{N}$, one has  $$\int_{\big(\alpha^*_{\Omega}(\omega_i\setminus N(\Omega, o)\big)\cap \big(\alpha^*_{\Omega}(\omega_j\setminus N(\Omega, o)\big)}\,d\xi=0$$ for $i\neq j$ and  \[\alpha^*_{\Omega}\bigg(\Big(\bigcup_{i=1}^{\infty} \omega_i\Big)\setminus N(\Omega, o)\bigg)=\bigcup_{i=1}^{\infty}  \alpha^*_{\Omega}(\omega_i\setminus N(\Omega, o)).\] 

We are ready to prove the following result. 

\begin{proposition}  Let  $G\in \mathcal{G}_I^0$ and $\lambda$ be a nonzero finite Borel measure on $\S^n$ which is absolutely continuous with respect to $\,d\xi$. Then the measures defined in \eqref{GDOCM def-log} is equivalent to the following formula: 
for each Borel set $\omega\subseteq\S^n$,  \begin{eqnarray}
\widetilde{C}_{G,\lambda}(\Omega,\omega) = \int_{{\mathrm{\alpha}_{\Omega}^*}(\omega)} r_\Omega(\xi)G_z(r_\Omega(\xi),\xi)\,d\lambda(\xi). \label{GDOCM def-log-g0}
\end{eqnarray}  
\end{proposition}

 \begin{proof} Let  $\Omega\in \mathcal{K}$.  It follows from \eqref{radius=0} and \eqref{dual-cone-relation-1} that $r_{\Omega}(\xi)=0$ for $\xi\in \alpha^*_{\Omega} \big(\S^n\cap N(\Omega, o)\big)$ up to sets of $\,d\xi$-measure (hence of $\lambda$-measure) $0$.  Thus,  both \eqref{GDOCM def-log} and \eqref{GDOCM def-log-g0} give $\widetilde{C}_{G,\lambda}(\Omega,\omega)=0$ for $\omega\subseteq \S^n\cap N(\Omega, o)$. On the other hand, if $\omega\subseteq \S^n\setminus N(\Omega, o)$,  \eqref{dual-2-3} implies that, up to sets of $\lambda$-measure $0$,   $ \alpha_{\Omega}^*(\omega)\subseteq N^*(\Omega, o)$ and of course $\alpha^*_{\Omega}(\omega)=\alpha_{\Omega}^*(\omega\setminus N(\Omega, o)).$ This further implies that  both \eqref{GDOCM def-log} and \eqref{GDOCM def-log-g0} give the same value of $\widetilde{C}_{G,\lambda}(\Omega,\omega)$ for any $\omega\subseteq \S^n\setminus N(\Omega, o)$. 

Finally, for any $\omega\subseteq \S^n$, one can take $\omega_1=\omega \cap N(\Omega, o)$ and $\omega_2=\omega\setminus  N(\Omega, o)$. Hence, $\omega=\omega_1\cup \omega_2$ and $\omega_1\cap \omega_2$ is an empty set. The additivity yields that both \eqref{GDOCM def-log} and \eqref{GDOCM def-log-g0} give the same value of $\widetilde{C}_{G,\lambda}(\Omega,\omega)$ for any $\omega\subseteq \S^n$.  \end{proof}

The weak convergence for $\widetilde{C}_{G,\lambda}(\Omega,\omega)$ is summarized in our next result. Its proof follows from an argument similar to those in \cite{BF19,GHXY18}. For completeness, a short proof is included here. 
\begin{proposition}\label{prop-c-g-1} Let $G\in \mathcal{G}_I^0$ and $\lambda$ be a nonzero finite Borel measure on $\S^n$ which is absolutely continuous with respect to $\,d\xi$. Then the measure $\widetilde{C}_{G,\lambda}(\cdot, \cdot)$ is weakly convergent on $\mathcal{K}$, namely, if $\Omega_i\in \mathcal{K}$ for all $i\in \mathbb{N}$ and $\Omega_i$ converges to $\Omega\in \mathcal{K}$ in the Hausdorff metric, then $\widetilde{C}_{G,\lambda}(\Omega_i,\cdot)\rightarrow \widetilde{C}_{G,\lambda}(\Omega,\cdot)$ weakly.
\end{proposition} 
\begin{proof}  First of all, we can prove that, for any bounded Borel function $g:\S^n\to \mathbb{R}$, 
\begin{equation}\label{variable-change-111} 
\int_{\S^n} g(u)\, d\widetilde{C}_{G,\lambda}(\Omega, u) = \int_{\S^n\cap N^*(\Omega, o)} I_{\Omega}(\xi)\,d\lambda(\xi)= \int_{\S^n} I_{\Omega}(\xi)\,d\lambda(\xi),  \end{equation} where  $I_{\Omega}(\xi)=g(\alpha_{\Omega}(\xi)) r_{\Omega}(\xi) G_z(r_{\Omega}(\xi), \xi). $  Indeed, this formula has been proved in \cite{HSYZ-21} for $\Omega\in \mathcal{K}_0$. When $\Omega\in \mathcal{K}$, it is well known that the $\,d\xi$-measure (and hence $\lambda$-measure) of the set $\S^n\cap \partial N^*(\Omega, o)$ is $0$. Let $\mathrm{int}N^*(\Omega, o)$ be the interior of $N^*(\Omega, o)$. If $g=\mathbf{1}_{\omega}$ be the characteristic function of Borel set $\omega\subseteq \S^n$, then \eqref{GDOCM def-log-g0} implies  \begin{equation} \int_{\S^n}  \mathbf{1}_{\omega}(u)\, d\widetilde{C}_{G,\lambda}(\Omega, u)= 
 \widetilde{C}_{G,\lambda}(\Omega, \omega)  =  \int_{{\mathrm{\alpha}_{\Omega}^*}(\omega\setminus N(\Omega, o))} r_\Omega(\xi)G_z(r_\Omega(\xi),\xi)\,d\lambda(\xi). \label{special-form--1}  \end{equation}  
 Recall that, for $\omega\subseteq \S^n$ and for $\,d\xi$-almost all $\xi\in \S^n\cap N^*(\Omega, o)$,  $\xi\in \alpha_{\Omega}^*(\omega)$ if and only if $\alpha_{\Omega}(\xi)\in \omega.$ Together with \eqref{dual-2-3} and \eqref{special-form--1},  one further gets  \begin{eqnarray*} \int_{\S^n}  \mathbf{1}_{\omega}(u)\, d\widetilde{C}_{G,\lambda}(\Omega, u)  &=& \int_{\alpha^*_{\Omega}(\omega)\cap N^*(\Omega, o)}  r_\Omega(\xi)G_z(r_\Omega(\xi),\xi)\,d\lambda(\xi)\\ &=&  \int_{\S^n \cap N^*(\Omega, o)} \mathbf{1}_{\omega}\big(\alpha_{\Omega}(\xi) \big) r_\Omega(\xi)G_z(r_\Omega(\xi),\xi)\,d\lambda(\xi)\\ &=&  \int_{\S^n} \mathbf{1}_{\omega}\big(\alpha_{\Omega}(\xi) \big) r_\Omega(\xi)G_z(r_\Omega(\xi),\xi)\,d\lambda(\xi). \end{eqnarray*}  
This claims that  \eqref{variable-change-111} holds for $g=\mathbf{1}_{\omega}$ (in view of \eqref{radius=0}). Hence, \eqref{variable-change-111} follows by a standard argument based on approaching a bounded Borel function $g$ by a sequence of bounded simple functions. 

Let $G\in \mathcal{G}_I^0$, $\Omega_i\in \mathcal{K}$ converge to $\Omega\in \mathcal{K}$ in the Hausdorff metric, and  $g:\S^n\rightarrow \mathbb{R}$ be a continuous function. It holds that $I_{\Omega_i}(\xi)\rightarrow I_{\Omega}(\xi)$ for $\,d\xi$-almost all $\xi\in \S^n$ (see \cite[Proposition 5.2]{GHXY18} when $\Omega\in \mathcal{K}$ and \cite[Proposition 6.2]{GHWXY18} when $\Omega\in \mathcal{K}_0$). Moreover, $$\sup\big\{I_{\Omega_i}(\xi): \ \xi\in \S^n \ \ \mathrm{and}\ i\in \mathbb{N} \big\}<\infty.$$ These, together with the dominated convergence theorem, immediately show that 
\begin{equation*} 
 \lim_{i\rightarrow \infty}  \int_{\S^n} I_{\Omega_i}(\xi)\,d\lambda(\xi)= \int_{\S^n} I_{\Omega}(\xi)\,d\lambda(\xi). \end{equation*} In view of \eqref{variable-change-111}, one gets 
 \begin{equation*} 
\lim_{i\rightarrow \infty} \int_{\S^n} g(u)\, d\widetilde{C}_{G,\lambda}(\Omega_i, u)   =\int_{\S^n} g(u)\, d\widetilde{C}_{G,\lambda}(\Omega, u) \end{equation*} holds for all continuous functions on $\S^n$. Hence,  $\widetilde{C}_{G,\lambda}(\Omega_i,\cdot)\rightarrow \widetilde{C}_{G,\lambda}(\Omega,\cdot)$ weakly. \end{proof} 

\subsection{Proof of Theorem \ref{main3}} \label{sec-main-3-1} 
{~}

Theorem \ref{main3} is now proved with the help of the $C^2$-estimates established in Lemmas \ref{lowerboundsigma_n} and \ref{boundsprincipleradii} in Section \ref{section-5}.  

\begin{proof} [Proof of Part $(\mathrm{ii})$ in Theorem \ref{main3}.] Following the notations in Section \ref{section-3}, for  $\Omega_0\in\mathcal{K}_0$, let 
$\mathcal{M}_0=\partial \Omega_0$ be a closed, smooth and uniformly convex hypersuface. The support function of $\Omega_0$ is $u_0$.   Denote by $r_{\Omega_0}$ the radial function of $\Omega_0$ and assume \eqref{ep}, namely $r_{\Omega_0}(\xi)\ge c_0\ge10\varepsilon$ for some constant $c_0>0$ and for small $\varepsilon \in (0, \delta)$ with $\delta\in (0, 1)$. Let $X_\varepsilon(\cdot,t)$ be the solution to \eqref{SF2}, $u_\varepsilon(\cdot,t)$ be the support function of $\Omega_t^\varepsilon$ and  $\mathcal{M}_t^\varepsilon=\partial \Omega_t^\varepsilon$. Let $T$ be the maximal time such that $u_\varepsilon(\cdot,t)$ is positive, smooth and uniformly convex for all $t\in[0,T)$.  

First of all, under the conditions on $f$, $p_\lambda$, $G$,  $\Psi$ and $\psi$ stated in  Part $(\mathrm{ii})$ of Theorem \ref{main3}, one can find a constant $\overline{C}_\varepsilon>0$ (depending on $\varepsilon$ but independent of $t$), such that, the principal curvature radii of $\mathcal{M}^{\varepsilon}_t$ are bounded from above and below. That is, for all $t\in [0, T)$, one has, 
\begin{eqnarray}\label{sigma lower bonud}
\overline{C}_\varepsilon^{-1} I\le\nabla^2u_{\varepsilon}(\cdot, t)+u_{\varepsilon}(\cdot, t)I\le \overline{C}_\varepsilon I.
\end{eqnarray} In fact, conditions \eqref{cond1}, \eqref{cond2} and \eqref{cond3} follow from Lemmas \ref{upperbound}-\ref{lowerbound}.   Hence,  \eqref{sigma lower bonud} is a direct consequence of Lemmas \ref{lowerboundsigma_n} and \ref{boundsprincipleradii}, where $u(\cdot, t)$, $\eta(t)$ and $\Phi(x,u,\nabla u)$ in Lemmas \ref{lowerboundsigma_n} and \ref{boundsprincipleradii} are replaced by $u_{\varepsilon} =u_{\varepsilon}(\cdot, t)$,  $\eta_\varepsilon(t)$, and respectively,    
\begin{eqnarray*} \Phi(x,u_{\varepsilon},\nabla u_{\varepsilon})&=&f(x)\widehat{\psi}_\varepsilon(u_{\varepsilon},x)(u_{\varepsilon}^2+|\nabla u_{\varepsilon}|^2)^\frac{n}{2} p^{-1}_\lambda\bigg(\frac{u_{\varepsilon}x+\nabla u_{\varepsilon}}{\sqrt{u_{\varepsilon}^2+|\nabla u_{\varepsilon}|^2}}\bigg)\\ &&\times G_z^{-1}\bigg(\sqrt{u_{\varepsilon}^2+|\nabla u_{\varepsilon}|^2},\frac{u_{\varepsilon} x+\nabla u_{\varepsilon}}{\sqrt{u_{\varepsilon}^2+|\nabla u_{\varepsilon}|^2}}\bigg).\end{eqnarray*}

In view of \eqref{sigma lower bonud} and Lemmas \ref{upperbound}-\ref{lowerbound}, the flow \eqref{uSF2} is uniformly parabolic and then $|\partial_tu_\varepsilon|_{L^\infty(\mathbb{S}^n\times[0,T))}\le \overline{C}_\varepsilon$ (abuse of notation $\overline{C}_\varepsilon$). It follows from the results of Krylov and Safonov\cite{KS81} that the  H\"{o}lder continuity estimates for $\nabla^2u_{\varepsilon}$ and $\partial_tu_{\varepsilon}$  can be obtained. Hence, the standard theory of the linear uniformly parabolic equations can be used to get the higher order derivatives estimates, which further implies the the long time existence of a positive, smooth and uniformly convex solution to the flow \eqref{uSF2}. Moreover, 
\begin{eqnarray}
\label{est1}&&1/C_\varepsilon\le u_\varepsilon(x,t)\le C\ \ \ \mathrm{for\ all}\ (x,t)\in\mathbb{S}^n\times[0,\infty),\\
\label{est2}&&1/C\le|\eta_\varepsilon(t)|\le C\ \ \ \mathrm{for\ all}\ t\in[0,\infty),\\
\label{est3}&&\nabla^2u_\varepsilon+u_\varepsilon I\ge1/C_\varepsilon\ \ \ \mathrm{on}\ \,\mathbb{S}^n\times[0,\infty),\\
\label{est4}&&|u_\varepsilon|_{C_{x,t}^{k,l}(\mathbb{S}^n\times[0,\infty))}\le C_{\varepsilon,k,l},
\end{eqnarray}
where $C_\varepsilon$ and $ C_{\varepsilon,k,l}$ are constants depending on $\varepsilon, G, \Psi, f, p_\lambda$ and  $\Omega_0$. Together with Lemma \ref{monotone1},  one has $0\leq \mathcal{J}_\varepsilon(u_{\varepsilon}(\cdot,t)) \le \mathcal{J}_\varepsilon(u_{0})$ for all $t\in[0,\infty)$, and hence   \begin{eqnarray*}  \int_{0}^{\infty} \Big|\frac{d}{dt}\mathcal{J}_\varepsilon(u_{\varepsilon}(\cdot,s))\Big| \,ds =\mathcal{J}_\varepsilon(u_{0})-
\lim_{t\to\infty}\mathcal{J}_\varepsilon(u_{\varepsilon}(\cdot,t))\leq \mathcal{J}_\varepsilon(u_{0}).
\end{eqnarray*} This further yields the existence of a sequence $t_i\to\infty$ such that
\[\frac{d}{dt}\mathcal{J}_\varepsilon(u_{\varepsilon}(\cdot,t_i))\to 0\ \ \ \mathrm{as}\,\,t_i\to\infty.\]

By virtue of \eqref{est1}-\eqref{est4} and Lemma \ref{monotone1}, $\{u_\varepsilon(\cdot, t_i)\}_{i\in \mathbb{N}}$ is uniformly bounded and then a subsequence of $\{u_\varepsilon(\cdot, t_i)\}_{i\in \mathbb{N}}$ (which will be still denoted by $\{u_\varepsilon(\cdot, t_i)\}_{i\in \mathbb{N}}$)  converges to a positive and uniformly convex function $u_{\varepsilon,\infty}\in C^\infty(\mathbb{S}^n)$ satisfying  that  \begin{equation}  {u_{\varepsilon,\infty}}(x) r_{\varepsilon,\infty}^{-n}(\xi) G_z\big(r_{\varepsilon,\infty}(\xi),\xi\big)p_\lambda(\xi) \det\!\big(\nabla^2u_{\varepsilon,\infty}(x)\!+\! u_{\varepsilon,\infty}(x)I\big)\!=\!\gamma_{\varepsilon}f(x)\widehat{\psi}_\varepsilon(u_{\varepsilon,\infty},x), \label{equation-infty-1}\end{equation} 
where $r_{\varepsilon,\infty}(\xi)=(u_{\varepsilon, \infty}^2(x)+|\nabla u_{\varepsilon,\infty}(x)|^2)^{1/2}$ and   $\gamma_\varepsilon$ is a constant given by
\[\gamma_\varepsilon=\lim_{t_i\to\infty}\frac{1}{\eta_\varepsilon(t_i)}=\frac{\int_{\S^n}r_{\varepsilon,\infty}(\xi)G_z\big(r_{\varepsilon,\infty}(\xi),\xi\big)p_\lambda(\xi)\,d\xi}{\int_{\S^n}f(x)\widehat{\psi}_\varepsilon(u_{\varepsilon,\infty},x)dx}.\] Clearly, $u_{\varepsilon,\infty}\geq 1/C_{\varepsilon}$ on $\S^n$ and then  $\Omega_{\varepsilon,\infty}\in\mathcal{K}_0$ (where $\Omega_{\varepsilon,\infty}$ is the convex body determined by $u_{\varepsilon,\infty}$). It follows from \eqref{Jac-1-1} and \eqref{equation-infty-1} that, for any Borel set $\omega\subseteq \mathbb{S}^n$,  \begin{eqnarray}\label{varepsilon eq}
\int_{\alpha_{\Omega_{\varepsilon,\infty}}^*(\omega)}r_{\varepsilon,\infty}(\xi)  G_z\big(r_{\varepsilon,\infty}(\xi) ,\xi\big)p_\lambda(\xi)\,d\xi=\gamma_\varepsilon\int_{\omega}f(x) \widehat{\psi}_\varepsilon(u_{\varepsilon,\infty},x)\,dx.
\end{eqnarray} Note that $\,d\mu(x)=f(x)\,dx$ and $\,d\lambda(\xi)=p_{\lambda}(\xi)\,d\xi$. Formula \eqref{GDOCM def-log} yields that \eqref{varepsilon eq} is equivalent to \eqref{GDOMP1} for $\Omega_{\varepsilon,\infty}$, namely,  \begin{eqnarray}\label{varepsilon eq-measure}
\gamma_{\varepsilon} \psi(u_{\varepsilon,\infty}, \cdot)\,d\mu =  \,d\widetilde{C}_{G,\lambda}(\Omega_{\varepsilon,\infty},\cdot). \end{eqnarray}

Recall that for $G\in \mathcal{G}_I^0$ and $\Omega\in \mathcal{K}$, one has \begin{eqnarray*} 
\widetilde{V}_{G,\lambda}(\Omega)=\int_{\mathbb{S}^n}G(r_\Omega(\xi),\xi)d\lambda(\xi).
\end{eqnarray*}  Lemma \ref{monotone1} then implies that $\widetilde{V}_{G,\lambda}(\Omega_0)=\widetilde{V}_{G,\lambda}(\Omega_t^{\varepsilon})$, and hence $$\widetilde{V}_{G,\lambda}(\Omega_0)=\lim_{i\rightarrow \infty} \widetilde{V}_{G,\lambda}(\Omega_{t_i}^{\varepsilon})=\widetilde{V}_{G,\lambda}(\Omega_{\varepsilon,\infty}).$$ That is, $\Omega_{\varepsilon,\infty}\in \mathcal{K}_V$ with  
\begin{eqnarray*}
\mathcal{K}_V=\Big\{\mathbb{K}\in\mathcal{K}: \widetilde{V}_{G,\lambda}(\mathbb{K})=\widetilde{V}_{G,\lambda}(\Omega_0)\Big\}.
\end{eqnarray*}  Together with \eqref{functional} and   Lemma \ref{monotone1}, one sees that $\Omega_{\varepsilon,\infty}$ solves the following optimization problem:  \[\inf\Big\{\int_{\S^n}f(x)\widehat{\Psi}_\varepsilon(u_\mathbb{K}(x), x)\,dx: \mathbb{K}\in\mathcal{K}_V\Big\}.\]  
    
It follows from Lemmas \ref{etaepsbound} and \ref{widthbound} that $\frac{1}{C}\le w_{\Omega_{\varepsilon,\infty}}^-\le w_{\Omega_{\varepsilon,\infty}}^+\le C$ and  $\frac{1}{C}\le\gamma_\varepsilon\le C$, respectively, where $C$ is independent of $\varepsilon$.  Hence, a constant $\gamma_0>0$ and  a sequence $\varepsilon_i\to 0$ can be found so that  $\gamma_{\varepsilon_i}\to\gamma_0$. Due to Proposition \ref{Prop-V-G-1} and the remark immedately before it, one can even assume that $\Omega_{\varepsilon_i,\infty}$ converges to a $\Omega_\infty\in\mathcal{K}_V$ in the Hausdorff metric. Note that, by Proposition \ref{prop-c-g-1},  $ \widetilde{C}_{G, \lambda}(\Omega_{\varepsilon_i,\infty})\rightharpoonup  \widetilde{C}_{G, \lambda}(\Omega_{\infty})$ weakly. It follows from \eqref{varepsilon eq}, \eqref{varepsilon eq-measure}, the fact that $\widehat{\psi}_{\varepsilon}(s,x)$ is bounded on $[0, C]\times \S^n$ and $\varepsilon\in (0, \delta)$, and the dominated convergence theorem that, for each Borel set $\omega\subseteq \S^n$,  
\begin{equation}\gamma_0 \int_{\omega} \psi(u_{\infty},x)\,d\mu(x)= 
\int_{\alpha_{\Omega_{\infty}}^*(\omega)}r_{\infty}(\xi) G_z(r_{\infty}(\xi),\xi)p_\lambda(\xi)\,d\xi = \int_{\omega}\, \,d\widetilde{C}_{G,\lambda}(\Omega_{\infty},\xi). \label{equ-solved-1} \end{equation}  Moreover, $\Omega_\infty$ satisfies
\begin{eqnarray}\label{opti-123}
\int_{\S^n}f\Psi(u_{\Omega_{\infty}},x)dx=\inf\Big\{\int_{\S^n}f(x)\Psi(u_\mathbb{K},x)dx: \mathbb{K}\in\mathcal{K}_V\Big\}.
\end{eqnarray}
Clearly, equation \eqref{equ-solved-1} can be reformulated as $$ \psi(u_{\infty},x)\,d\mu(x)=  \tau \,d\widetilde{C}_{G,\lambda}(\Omega_{\infty},\xi) \ \ \mathrm{with}\ \  \tau=\frac{1}{\gamma_0}=  \frac{ \int_{\S^n} \psi(u_{\infty},x)\,d\mu(x)} {\int_{\S^n} \,d\widetilde{C}_{G,\lambda}(\Omega_{\infty},\xi)}. $$
 Thus, $\Omega_{\infty}$ satisfies \eqref{GDOMP1}, and this concludes the proof of Part $(\mathrm{ii})$ in Theorem \ref{main3}. \end{proof}

\begin{proof}[Proof of Part $(\mathrm{i})$ in Theorem \ref{main3}.] In this case, we assume that \eqref{good-condition-1}  holds, i.e., 
\[\liminf_{s\to0^+}\frac{sG_z(s,x)}{\psi(s,x)}=\infty \ \ \mathrm{for\ all}\ x\in\S^n.\]   
It follows from \eqref{upper11}-\eqref{etabound1} and \eqref{lower11} that \begin{eqnarray}\label{C2esti1}
C^{-1}I\le\nabla^2u+uI\le CI. 
\end{eqnarray} This is  a direct consequence of Lemmas \ref{lowerboundsigma_n} and \ref{boundsprincipleradii} by letting  
\begin{eqnarray*} \Phi(x,u,\nabla u)\!=\!f(x)\psi(u,x)(u^2\!\!+\!|\nabla u|^2)^\frac{n}{2} p^{-1}_\lambda\bigg(\!\!\frac{ux \!+\!\nabla u}{\sqrt{u^2\!\!+\!|\nabla u|^2}}\!\bigg)  G_z^{-1}\!\bigg(\!\! \sqrt{u^2\!\!+\!|\nabla u|^2},\frac{ux \!+\!\nabla u}{\sqrt{u^2\!\!+\!|\nabla u|^2}}\!\bigg).\end{eqnarray*}  
Following along the same lines as those in the proof of Part $(\mathrm{ii})$ in Theorem \ref{main3}, $u(\cdot,t)$ remains a positive, smooth and uniformly convex solution of the flow \eqref{uSSF} for all time $t>0$, due to the above priori estimates \eqref{upper11}-\eqref{etabound1}, \eqref{lower11} and \eqref{C2esti1}.  By Lemma \ref{mono},  
\begin{eqnarray*}
\frac{d}{dt}\mathcal{J}(u(\cdot,t))\le0\ \ \ \mathrm{for\ all} \  t\in[0, \infty).
\end{eqnarray*}
Hence, a sequence $t_i\to\infty$ can be found so that 
\[\frac{d}{dt}\mathcal{J} (u(\cdot,t_i))\to 0\ \ \ \mathrm{as}\,\,t_i\to\infty,\] and  $u(\cdot,t_i)$ converges smoothly to a positive, smooth and uniformly convex function $u_\infty$ solving \eqref{elliptic}, where 
\[\gamma=\displaystyle\lim_{t_i\to\infty}\frac{1}{\eta(t_i)}=\frac{\int_{\S^n}r_{\infty}(\xi)G_z\big(r_{\infty}(\xi),\xi\big)p_\lambda(\xi)\,d\xi}{\int_{\S^n}f(x){\psi}(u_{\infty},x)dx}\] with $r_{\infty}=r_{\Omega_{\infty}}$  the radial function of the convex body $\Omega_{\infty}\in \mathcal{K}_0$ whose support function is $u_{\infty}$. In particular, $\Omega_{\infty}$ satisfies  \eqref{GDOMP} with $\tau=1/\gamma$ and \eqref{opti-123}.
\end{proof}

\subsection{Proof of Theorem \ref{main1} } \label{sec-main-1-1}{~}

Using the  standard approximations for the functions $G$, $p_\lambda$ and $\Psi$ in Theorem \ref{main1}, the following result can be obtained by the proof of Theorem \ref{main3} directly. 
\begin{cor}\label{Part ii}
Let $G$, $p_\lambda$ and $\Psi$ be as in Theorem \ref{main1} and $f$ be a smooth positive function on $\S^n$. There exist $\gamma>0$ and $\Omega\in\mathcal{K}_V$ such that  
\begin{align*} \int_{\alpha_\Omega^*(\omega)}rG_z(r,\xi)p_\lambda(\xi)d\xi &=\gamma\int_{\omega}f\psi(u,x)dx,\ \ \mathrm{for\ all\ Borel\ set}\  \  \omega\subseteq \S^n, \\ \int_{\S^n}f\Psi(u_\Omega,x)\,dx &=\inf\Big\{\int_{\S^n}f\Psi(u_\mathbb{K},x)dx: \mathbb{K}\in\mathcal{K}_V\Big\}.\end{align*}
\end{cor}

The proof of the following result can be found in e.g., \cite[Lemma 3.7]{CL19}.  
\begin{lem}\label{mu}
Let $\mu$ be a non-zero finite Borel measure on $\mathbb{S}^n$. There is a sequence of positive smooth functions $f_j$ defined on $\mathbb{S}^n,$ such that $\,d\mu_j=f_j\,d\xi$  converges to $\mu$ weakly.
\end{lem}

\begin{proof}[Proof of Theorem \ref{main1}]
Let $f_j$ be a sequence of positive and smooth functions on $\mathbb{S}^n$ such that the measures $\mu_j$ converge to $\mu$ weakly as $j\to\infty$. By  Corollary \ref{Part ii},  there are $\gamma_j>0$ and $\Omega_j\in\mathcal{K}_V$
such that
\begin{eqnarray}\label{gdomp}
\int_{\alpha^*_{\Omega_{j}(\omega)}}r_{\Omega_j} (\xi) G_z(r_{\Omega_j},\xi)d\lambda(\xi)&=&\gamma_j\int_{\omega}f_j \psi(u_{\Omega_j},x)dx 
\end{eqnarray}  for any Borel set $\omega\subseteq \mathbb{S}^n$, and \begin{eqnarray}\label{kv}
\int_{\mathbb{S}^n}f_j\Psi(u_{\Omega_j},x)dx&=&\inf\Big\{\int_{\mathbb{S}^n}f_j\Psi(u_\mathbb{K},x)dx:\mathbb{K}\in\mathcal{K}_V\Big\}.
\end{eqnarray}
Clearly, the constants $\gamma_j$ can be calculated by 
\begin{eqnarray}\label{constant-g-j} \gamma_j=\frac{\int_{\S^n}r_{\Omega_j} (\xi) G_z(r_{\Omega_j},\xi)d\lambda(\xi)}{\int_{\S^n}f_j \psi(u_{\Omega_j},x)dx}.
 \end{eqnarray} 

Let ${u_j}_{\max}=\displaystyle\max_{x\in\S^n}u_{{\Omega}_j}(x)$ and $\bar{x}_j\in\mathbb{S}^n$ satisfy that  ${u_j}_{\max}=u_{{\Omega}_j}(\bar{x}_j)$. Then
\begin{eqnarray*}
u_{{\Omega}_j}(x)\ge \langle x, \bar{x}_j\rangle {u_j}_{\max} \ \ \ \ \mathrm{for\ all}\ x\in\mathbb{S}^n\ \mathrm{with}\ \langle x, \bar{x}_j\rangle >0.
\end{eqnarray*}
Denote by $\mathbb{B}^{n+1}_t$  the origin-symmetric Euclidean ball of $\R^{n+1}$ with radius $t>0$. Note that
\begin{eqnarray*}
\widetilde{V}_{G,\lambda}(\mathbb{B}^{n+1}_t)=\int_{\mathbb{S}^n}G(t,\xi)p_\lambda(\xi)d\xi,
\end{eqnarray*} and the function $t \mapsto \widetilde{V}_{G,\lambda}(\mathbb{B}^{n+1}_t)$ is continuous. As  $G_z>0$ on $(0,\infty)\times\mathbb{S}^n$, it follows from $\Omega_0\in \mathcal{K}_0$ that  there exists  $t_0>0$ such that $\widetilde{V}_G(\mathbb{B}^{n+1}_{t_0})=\widetilde{V}_{G,\lambda}(\Omega_0)$. That is, $\mathbb{B}^{n+1}_{t_0}\in\mathcal{K}_V$, which further yields, by \eqref{kv}, that for all $j\in \mathbb{N}$, 
\begin{eqnarray}\label{le}
\int_{\mathbb{S}^n}f_j\Psi({u_{\mathbb{B}^{n+1}_{t_0}}},x)dx \ge\int_{\mathbb{S}^n}f_j\Psi(u_{\Omega_j},x)dx.
\end{eqnarray}
Since $f_jdx\rightarrow \,d\mu$ weakly and $\mu$ is not concentrated on any closed hemisphere, there exists a constant  $c_{\mu}>0$, such that
\begin{eqnarray}\label{NCHj}
\int_{\S^n}\langle x,\bar{x}_j\rangle_{+}f_jdx>c_{\mu} \ \ \ \mathrm{for\ all}\ j\in \mathbb{N},
\end{eqnarray} where $a_+=\max\{a, 0\}$ for $a\in \mathbb{R}$. To this end, assume the contrary, there exists a subsequence of $j$ (which will still be denoted by $j$) such that {$\bar{x}_j\rightarrow x_0\in \mathbb{S}^n$} and 
\begin{eqnarray*}
\int_{\S^n}\langle x,x_0\rangle_{+}d\mu= \lim_{j\rightarrow \infty} \int_{\S^n}\langle x,\bar{x}_j\rangle_{+}f_jdx = 0.
\end{eqnarray*}
Consequently, one gets 
\begin{eqnarray*}
0=\int_{\S^n}\langle x,x_0\rangle_{+}d\mu \ge\frac{\mu(\{x\in\S^n: \langle x,  x_0\rangle>1/N\})}{N}. 
\end{eqnarray*} This  implies $\mu(\{x\in\S^n: \langle x, x_0\rangle>1/N\})=0$ for all $N>1$, which contradicts with, after taking $N\rightarrow \infty$,  the fact that $\mu$ is not concentrated on any closed hemisphere. 

 By  \eqref{NCHj},  for any $\zeta\in (0, 1)$, one has 
\begin{eqnarray*}
c_{\mu} \le\int_{\S^n}\langle x,\bar{x}_j\rangle_+f_jdx
&=&\int_{\{x\in\S^n:  \langle x, \bar{x}_j\rangle>\zeta\}}\langle x,\bar{x}_j\rangle_+f_jdx+\int_{\{x\in\S^n:   \langle x, \bar{x}_j\rangle\leq \zeta\}}\langle x,\bar{x}_j\rangle_+f_jdx
\\&\le&\int_{\{x\in\S^n:  \langle x, \bar{x}_j\rangle>\zeta\}}f_jdx+\zeta \int_{\S^n} f_jdx,
\end{eqnarray*} since $0<\langle x,\bar{x}_j\rangle_+<1$.  
Hence
\begin{eqnarray*}
\int_{\{x\in\S^n:  \langle x, \bar{x}_j\rangle>\zeta\}}f_jdx \ge c_{\mu}-\zeta \int_{\S^n}f_jdx.
\end{eqnarray*}
As $\,d\mu_j= f_jdx$ converges to $\mu$ weakly, one can choose $\zeta_1 \in (0, 1)$ to be a small enough constant such that, for all $j\in \mathbb {N}$,   $$\int_{\{x\in\S^n:  \langle x, \bar{x}_j\rangle>\zeta_1\}}f_jdx\geq c_{\mu}-\zeta_1  \int_{\S^n}f_jdx>\frac{c_{\mu}}{2}.$$ 
Consequently, it follows that
\begin{eqnarray}\label{ge}
\int_{\mathbb{S}^n}\Psi(u_{\Omega_j},x)f_jdx&\ge& \int_{\{x\in\S^n:  \langle x, \bar{x}_j\rangle>\zeta_1\}} \Psi({u_j}_{\max}  \langle x, \bar{x}_j\rangle,x)f_jdx
\nonumber\\&\ge&\min_{x\in\S^n} \Psi(\zeta_1 {u_j}_{\max},x) \int_{\{x\in\S^n:  \langle x, \bar{x}_j\rangle>\zeta_1\}}f_jdx \nonumber\\&\ge&\frac{c_{\mu}}{2} \cdot \min_{x\in\S^n}\Psi(\zeta_1 {u_j}_{\max},x).
\end{eqnarray}

As $\,d\mu_j= f_jdx$ converges to $\mu$ weakly, \eqref{le}  implies that for all $j\in \mathbb{N}$, 
\begin{eqnarray*} \int_{\mathbb{S}^n}f_j\Psi(u_{\Omega_j},x)dx\leq \sup_{j\in \mathbb{N}}
\int_{\mathbb{S}^n}f_j\Psi(u_{{\mathbb{B}^{n+1}_{t_0}}},x)dx <\infty.
\end{eqnarray*} Together with \eqref{ge} and the fact that  $\Psi(s, \xi)\to\infty$ for each $\xi\in\mathbb{S}^n$ as $s\to\infty$, one gets  that $\{{u_j}_{\max}\}_{j\in \mathbb{N}}$ is a bounded sequence.  The Blaschke selection theorem implies the existence of a subsequence $\{\Omega_j\}_{j\in \mathbb{N}}$ (which will still be denoted by $\{\Omega_j\}_{j\in \mathbb{N}}$) and a compact convex set $\Omega$ such that $\Omega_j\rightarrow \Omega.$ According to Proposition \ref{Prop-V-G-1}, one has $$\widetilde{V}_{G, \lambda}(\Omega)= \lim_{i\rightarrow \infty}  \widetilde{V}_{G, \lambda}(\Omega_i)=\widetilde{V}_{G,\lambda}(\Omega_0)>0.$$  This proves $\Omega\in \mathcal{K}$ and  hence $\Omega\in \mathcal{K}_V.$ It follows from \eqref{gdomp}, \eqref{constant-g-j}, Proposition \ref{prop-c-g-1} and the dominated convergence theorem that $\gamma_j\to\gamma$ and \begin{eqnarray*}
\int_{\alpha^*_{\Omega(\omega)}}r_{\Omega} (\xi) G_z(r_{\Omega},\xi)d\lambda(\xi)&=&\gamma \int_{\omega}f \psi(u_{\Omega},x)dx 
\end{eqnarray*}  holds for each Borel set $\omega\subseteq \mathbb{S}^n$. This  concludes that $\Omega$ satisfies \eqref{GDOMP1} as desired.
\end{proof}

\section{Appendix: $C^2$-estimates} \label{section-5}
This section is devoted to the second order derivative estimates. Such estimates were obtained in \cite[Section 8]{CL19}.  For readers' convenience and for completeness, a brief proof is included. 
Abuse of notations, $C$ and $C_0$ are some constants independent of $(x, t)\in \S^n\times [0, T)$. 
\begin{lem}\label{lowerboundsigma_n} Let $\Phi(x,s,y): \S^n\times[0,\infty)\times\R^n\to[0,\infty)$ be a smooth function such that $\Phi(x,s, z)>0$ whenever $s>0$. Let $T>0$ be a constant  and $u(\cdot, t)$ be a positive, smooth and uniformly convex solution to
\begin{eqnarray}\label{geneflow}
\frac{\partial u}{\partial t}(x,t)=-\Phi(x,u,\nabla u)\big(\det(\nabla^2 u+uI)\big)^{-1}+u\eta(t) \ \ \ \mathrm{on} \ \S^n\times [0, T).
\end{eqnarray}  If there is a  positive constant $C_0$ such that,  for all $x\in \S^n$ and all $t\in [0,T)$,  
\begin{eqnarray}
\label{cond1}1/C_0\le u(x,t)\le C_0, \\
\label{cond2}|\nabla u|(x,t)\le C_0, \\
\label{cond3}0\le|\eta(t)|\le C_0,
\end{eqnarray} then the following statement holds: for all $(x, t)\in \S^n\times [0, T)$, 
\begin{eqnarray}\label{lower bound sigma}
\det(\nabla^2u+uI)\ge1/C,
\end{eqnarray}
where $C>0$ is a constant depending on $\Omega_0$ (the initial hypersurface), $C_0$, $|\Phi|_{L^\infty(D)}$, $|1/\Phi|_{L^\infty(D)}$, $|\Phi|_{C^1_{x,s,p}(D)}$ with $D=\S^n\times[1/C_0,C_0]\times \mathbb B^n_{C_0}$.
\end{lem}
\begin{proof} Let  ${\epsilon}=\frac{1}{2}\inf_{{\S^n\times{[0,T)}}}u(x, t)$. Note that $\epsilon\geq \frac{1}{2C_0}$ due to \eqref{cond1}.  Consider the following auxiliary function
\begin{eqnarray} Q=\frac{-u_t+u\eta (t)}{u-\epsilon}=\frac{\Phi(x,u,\nabla u)\big(\det(\nabla^2 u+uI)\big)^{-1}}{u-\epsilon}. \label{Q-u--1} \end{eqnarray}
Denote by $Q_i=\nabla_i Q$  the $i$-th component of $\nabla Q$. Let $x_t\in \S^n$ for each $t\in [0, T)$ be such that $Q(x_t,t)=\max_{x\in \mathbb{S}^n}Q(x,t).$ 

 At $(x_t, t)$, one gets 
\begin{eqnarray}\label{q_1}
0&=&Q_i=\frac{-u_{ti}+u_i \eta (t)}{u-\epsilon}-\frac{(-u_t+ u\eta(t)){u_i}}{(u-\epsilon)^2}=\frac{-u_{ti}}{u-\epsilon}+\frac{(\eta(t)-Q){u_i}}{u-\epsilon},\\ 
 \label{q_2}
0&{\ge}&{Q_{ij}}=\frac{-u_{tij}+ \eta u_{ij}}{u-\epsilon}-\frac{(-u_t+ \eta u){u_{ij}}}{(u-\epsilon)^2}=\frac{-u_{tij}}{u-\epsilon}+\frac{(\eta -Q) u_{ij}}{u-\epsilon}.
\end{eqnarray} That is, the matrix $(Q_{ij})$ is a negative definite matrix.  As $u-\epsilon>0$, one gets 
\begin{eqnarray}\label{q*2} u_{ti}=(\eta(t)-Q){u_i} \ \ \ \mathrm{and} \ \ \ u_{tij}\ge(\eta-Q) u_{ij}. 
\end{eqnarray}  By \eqref{Q-u--1}, \eqref{q_1} and \eqref{q_2},  one has 
\begin{eqnarray}\label{Phi_t}
\qquad\partial_t\Phi=\Phi_u u_t+\Phi_{u_k}u_{kt}=\Phi_u[\eta u-Q(u-\epsilon)]+\Phi_{u_k}u_k(\eta-Q)\le C'(1+Q),
\end{eqnarray} where $C'>0$ is a constant. 
Let $b_{ij}=u_{ij}+u\delta_{ij}$ and $b^{ij}$ be the inverse matrix of $b_{ij}$. It can be checked from  \eqref{Q-u--1} that 
  \begin{eqnarray}\label{Q_t}
{\partial}_{t}{Q}&=&\frac{\left(\Phi(x,u,\nabla u)\big(\det(\nabla^2 u+uI)\big)^{-1}\right)_t}{u-\epsilon}-Q\frac{u_t}{u-\epsilon}\nonumber \\&=& \frac{\big(\det(\nabla^2 u+uI)\big)^{-1}\partial_t\Phi}{u-\epsilon}-Qb^{ij}(u_{tij}+u_t\delta_{ij})-Q\frac{u_t}{u-\epsilon}\nonumber \\&=& \frac{Q}{\Phi} \partial_t\Phi -Qb^{ij}(u_{tij}+u_t\delta_{ij})-Q\frac{u_t}{u-\epsilon}.
\end{eqnarray} This further implies,  together with  \eqref{Q-u--1},  \eqref{q*2},  \eqref{Phi_t} and \eqref{Q_t}, that
\begin{eqnarray*}
\partial_tQ  &\le& \frac{C'(1+Q)Q}{\Phi}-Qb^{ij}\Big(\!(\eta-Q)u_{ij}+[\eta u-Q(u-\epsilon)]\delta_{ij}\!\Big)-\Big(\frac{\eta u}{u-\epsilon}-Q\Big)Q
\\  &\le& C^{''}(1+Q)Q-Qb^{ij}\Big((\eta-Q)b_{ij}+\epsilon Q\delta_{ij}\Big)-\Big(\frac{\eta u}{u-\epsilon}-Q\Big)Q
\\\  &=& \Big(C^{''}-n\eta-\frac{\eta u}{u-\epsilon}\Big)Q+(C^{''}+n+1)Q^2-\epsilon Q^2\sum_ib^{ii},
\end{eqnarray*} where $C^{''}$ is a positive constant. Without loss of generality, let $Q>1$ at $(x_t, t)$. Hence $Q<Q^2$ at $(x_t, t)$, and there exists two positive constants $C_1$ and $C_2$ such that \begin{eqnarray*}
\partial_tQ\le C_1Q^2-C_2Q^{2+\frac{1}{n}},
\end{eqnarray*} where we have used  \eqref{Q-u--1} and the inequality $$\sum_ib^{ii} \ge n\big(\det(\nabla^2 u+uI)\big)^{-\frac{1}{n}}= n \Big(\frac{(u-\epsilon)Q}{\Phi} \Big)^{\frac{1}{n}}.$$ Consequently,  $Q\le C_3$ for some constant $C_3>0.$ This, together with \eqref{Q-u--1}, implies \eqref{lower bound sigma}  with $C>0$ a constant depending on $C_0$, $|\Phi|_{L^\infty(D)}$, $|1/\Phi|_{L^\infty(D)}$, $|\Phi|_{C^1_{x,s,p}(D)}$ and the initial hypersurface $\Omega_0$.
\end{proof}

The following lemme provides the boundedness (from both above and below) of the principal curvature radii of $\mathcal{M}_t$. 
\begin{lem}\label{boundsprincipleradii} Let $\Phi(x,s,y)$ and $T$ be as stated in Lemma \ref{lowerboundsigma_n}. Let  $u(x, t)$ be a positive, smooth and uniformly convex solution to \eqref{geneflow} and satisfy conditions  \eqref{cond1}-\eqref{cond2}. Then, for all $t\in[0,T)$ and $x\in \S^n$, one has
\begin{eqnarray}\label{upperlower}
(\nabla^2u+uI)(x,t)\le C I \ \ \ \mathrm{for\ all}\ (x, t)\in \S^n\times [0, T),
\end{eqnarray}
   where $C$ is a positive constant depending on $C_0$, $u(\cdot,0)$, $|\Phi|_{L^\infty(D)}$, $|1/\Phi|_{L^\infty(D)}$, $|\Phi|_{C^1_{x,s,p}(D)}$ and $|\Phi|_{C^2_{x,s,p}(D)}$ with  $D=\S^n\times[1/C_0,C_0]\times \mathbb{B}^n_{C_0}$.
\end{lem}
\begin{proof}
Let $b_{ij}=u_{ij}+u\delta_{ij}$ and $b^{ij}$ be the inverse matrix of $b_{ij}$.  Let $$r(x, t)=\sqrt{u^2(x, t)+|\nabla u|^2(x, t)}$$ and consider the following auxiliary function
\begin{eqnarray*}
W(x, t)=\log b(x,t) -\beta \log u(x,t)+\frac{A}{2}r^2(x,t),
\end{eqnarray*}
where  $\beta $ and $A$ are large constants to be decided, and
\[b(x,t)=\max\bigg\{\sum b_{ij}(x,t)\zeta_i\zeta_j:  \sum_i \zeta_i^2=1\bigg\}.\]
Let $T'\in(0,T)$  be an arbitrary number but fixed. Assume that $W$ attains its maximum on $\S^n\times[0,T']$ at $(x_0,t_0)$ with $t_0>0$ (as otherwise, there is nothing to prove). By a rotation, we may assume that $b^{ij}$ and $b_{ij}$ are diagonal at $(x_0, t_0)$, and   $b(x_0,t_0)= b_{11}(x_0, t_0)$.  Let
\begin{eqnarray*} 
w(x,t)=\log b_{11}-\beta \log u+\frac{A}{2}r^2.
\end{eqnarray*}
Then $\max_{\S^n\times[0,T']}W=\max_{\S^n\times[0,T']}w$, and so $w$ achieves its maximum at $(x_0,t_0)$. In the following,  we will provide an upper bound for $w$ (independent of $T'$), and this is sufficient for \eqref{upperlower}, as $T'$ is arbitrary.

At  $(x_0, t_0)$, one has 
\begin{eqnarray}\label{partial t}
0 \!\!\!&\le& \!\!\!\partial_t w=b^{11}\partial_t b_{11}-\beta\frac{u_t}{u}+Arr_t,\\ 
 \label{partial1}
0 \!\!\!&=& \!\!\!\nabla_iw=b^{11}\nabla_ib_{11}-\beta \frac{u_i}{u}+Arr_i, \\ 
 \label{partial2}
\qquad\qquad0 \!\!\!&\ge& \!\!\! \nabla^2_{ii}w=b^{11}\nabla^2_{ii}b_{11}-(b^{11})^2\nabla_ib_{11}\nabla_ib_{11}
-\beta \frac{\nabla^2_{ii}u}{u}+\beta \frac{u_iu_i}{{u}^2}+A(rr_{ii}+r_ir_i).
\end{eqnarray}

Note that  
\begin{eqnarray}
\label{d*11}&r_t=&\frac{uu_t+\sum_iu_iu_{it}}{r},\\
\label{d*12}
&r_k=&\frac{uu_k+\sum_iu_iu_{ik}}{r}=\frac{u_kb_{kk}}{r},\\
\label{d*13}
&r_{kl}=&\frac{uu_{kl}+u_ku_l+\sum_iu_iu_{ikl}+\sum_iu_{ik}u_{il}}{r}-\frac{u_ku_lb_{kk}b_{ll}}{r^3}.
\end{eqnarray}

The equation \eqref{geneflow} can be rewritten as
\begin{equation}\label{log}
\log(\eta u-u_t)+\log\big(\det(\nabla^2 u+uI)\big)=\overline{\Phi}(x,u,\nabla u),
\end{equation}
where $\overline{\Phi}(x,u,\nabla u):=\log\left(\Phi(x,u,\nabla u) \right)$. Differentiating \eqref{log} gives
\begin{eqnarray}
\frac{\eta u_k- u_{tk}}{\eta u-u_t }=-\sum b^{ii}\nabla_kb_{ii} +\nabla_k\overline{\Phi}=-\sum b^{ii}u_{kii}-\sum b^{ii}u_i\delta_{ik}+\nabla_k \overline{\Phi},\nonumber \\
\label{d2}\qquad\qquad\frac{\eta u_{11}-u_{t11}}{\eta u-u_t}-\frac{(\eta u_1- u_{t1})^2}{(\eta u-u_t )^2}=-\sum b^{ii}\nabla^2_{11}b_{ii}+\sum b^{ii}b^{jj}(\nabla_1b_{ij})^2+\nabla^2_{11}\overline{\Phi}.
\end{eqnarray}
Dividing \eqref{partial t} by $\eta u-u_t$ (which is positive due to \eqref{geneflow}), by \eqref{d2} and $b_{11}=u_{11}+u$, we have
\begin{eqnarray*}
0&\le&-b^{11}\left(\frac{\eta u_{11}-u_{11t}}{\eta u-u_t}-\frac{\eta b_{11}}{\eta u-u_t}+1\right)-\frac{\beta u_t}{u(\eta u-u_t)}+\frac{Arr_t}{\eta u-u_t}
\\&=&-b^{11}\frac{\eta u_{11}-u_{11t}}{\eta u-u_t} -b^{11}+\eta\frac{1-\beta}{\eta u-u_t}+\frac{\beta}{u}+\frac{Arr_t}{\eta u-u_t}
\\&\le&b^{11}\sum b^{ii}\nabla^2_{11}b_{ii}-b^{11}\sum b^{ii}b^{jj}(\nabla_1b_{ij})^2-b^{11}\nabla^2_{11}\overline{\Phi} +\eta\frac{1-\beta}{\eta u-u_t}+\frac{\beta}{u}+\frac{Arr_t}{\eta u-u_t}.
\end{eqnarray*}  Together with \eqref{d*11}, \eqref{d*12},  \eqref{d*13} and employing \eqref{partial2}, the following holds: 
\begin{eqnarray}\label{inequ}
\nonumber0&\le& b^{11}\sum\Big(b^{11}b^{ii}(\nabla_ib_{11})^2-b^{ii}b^{jj}(\nabla_1b_{ij})^2\Big)+(nb^{11}-\sum b^{ii})+n\frac{\beta}{u}-\beta \sum b^{ii}
\\\nonumber&{}&-\beta\sum b^{ii}\frac{u_i^2}{u^2}-b^{11}\nabla^2_{11}\overline{\Phi} +\eta\frac{1-\beta}{\eta u-u_t}+\frac{\beta}{u}+\frac{Arr_t}{\eta u-u_t}-A\sum b^{ii}(rr_{ii}+r_ir_i)
\\\nonumber&\le&(nb^{11}-\sum b^{ii})+(n+1)\frac{\beta}{u}-\beta \sum b^{ii}-\beta\sum b^{ii}\frac{u_i^2}{u^2}-b^{11}\nabla^2_{11}\overline{\Phi}+\eta\frac{1-\beta}{\eta u-u_t}
\\\nonumber&{}&+\left(\frac{A\eta r^2}{\eta u-u_t}-Au+A\sum b^{ii}u_ku_{kii}+Ab^{ii}u_i^2-Au_k\nabla_k\overline{\Phi}\right)
\\\nonumber&{}&-A\sum b^{ii}(uu_{ii}+u_i^2+u_{ii}^2+u_ku_{kii})
\\\nonumber&\le&nb^{11}+(n+1)\frac{\beta}{u}-\beta \sum b^{ii} -b^{11}\nabla^2_{11}\overline{\Phi}+\eta\frac{Ar^2+1-\beta}{\eta u-u_t}+(n-1)Au
\\&{}&-Au_k\nabla_k\overline{\Phi}-A\sum b_{ii},
\end{eqnarray} where the following formula, by the Ricci identity, has been used: \begin{eqnarray*}
\nabla^2_{11}b_{ij}=\nabla^2_{ij}b_{11}-\delta_{ij}b_{11}+\delta_{11}b_{ij}+\delta_{1i}b_{1j}-\delta_{1j}b_{1i}.
\end{eqnarray*} It follows from  \eqref{partial1} that
\begin{eqnarray*}
-b^{11}\nabla_{11}^2\overline{\Phi} -Au_k\nabla_k \overline{\Phi} \!\!&\le&\!\!-b^{11}\Big(\frac{\Phi_{x_1x_1}}{\Phi}-\frac{(\Phi_{x_1})^2}{\Phi^2}+\frac{\Phi_{uu}u_1^2+\Phi_uu_{11}}{\Phi} -\frac{(\Phi_u u_1)^2}{\Phi^2}-\frac{(\Phi_{u_1}u_{11})^2}{\Phi^2}
\\ \!\!&{}&\!\!+\frac{\Phi_{u_1u_1}u^2_{11}
+\Phi_{u_i}u_{i11}}{\Phi}\Big)
 -Au_k\Big(\frac{  \Phi_{x_k}}{\Phi}+\frac{\Phi_u}{\Phi}u_k+\frac{\Phi_{u_i}}{\Phi}u_{ik}\Big).\end{eqnarray*}
 Hence, a constant $C_1>0$ can be found so that 
 \begin{eqnarray*}
 nb^{11}+(n+1)\frac{\beta}{u}+(n-1)Au
-b^{11}\nabla_{11}^2\overline{\Phi} -Au_k\nabla_k \overline{\Phi} &\le& C_1b^{11}+\beta C_1+AC_1.\end{eqnarray*}
Choose $\beta$ to be a constant such that $\beta>\max\{1+Ar^2,C_1\}$. Then  $C_1b^{11}-\beta\sum b^{ii}<(C_1-\beta)b^{11}<0$ and  \eqref{inequ} implies  \begin{eqnarray*} 
0&\le&-\beta \sum b^{ii}+\eta\frac{Ar^2+1-\beta}{\eta u-u_t}
  -A\sum b_{ii}+C_1 b^{11}+C_1 +C_1 b_{11}+\beta C_1+AC_1\\ &\le&-A\sum b_{ii}+\beta C_1+AC_1\\ &\le& -A b_{11}+\beta C_1+AC_1. 
\end{eqnarray*} This gives   $b_{11}\le C_2$ at  $(x_0, t_0)$ for a constant   $C_2>0$. Consequently, $w(x, t)$ is bounded from above by a constant and then $\max_{(x, t)\in \S^n\times [0, T)}b_{11}<C$ holds for some constant $C>0$ which depends only on $C_0$, $u(\cdot,0)$, $|\Phi|_{L^\infty(D)}$, $|1/\Phi|_{L^\infty(D)}$, $|\Phi|_{C^1_{x,s,p}(D)}$ and $|\Phi|_{C^2_{x,s,p}(D)}$.   \end{proof}

\noindent{\bf{Acknowledgement.} } The research of Li has been supported by NSFC (No. 12031017). The research of Sheng has been supported by NSFC (Nos. 12031017 and 11971424). The research of Ye has been supported by an NSERC grant, Canada. The research of Yi has been supported by Scientific Research Foundation for Scholars of HZNU (No. 4085C50220204091). The authors are greatly indebted to the reviewer for many valuable comments which greatly improve the presentation of the paper.

 \noindent Qi-Rui Li, \ \ {\tt qi-rui.li@zju.edu.cn} \\ 
School of Mathematical Sciences, Zhejiang University, Hangzhou 310027, China.
 
\vskip 2mm \noindent  Weimin Sheng, \ \ {\tt weimins@zju.edu.cn}\\ 
 School of Mathematical Sciences, Zhejiang University, Hangzhou 310027, China.
 
 \vskip 2mm \noindent Deping Ye, \ \ {\tt deping.ye@mun.ca}\\
  Department of Mathematics and Statistics, Memorial
University of Newfoundland, St. John's, Newfoundland A1C 5S7, Canada.  

 \vskip 2mm \noindent Caihong Yi, \ \ {\tt caihongyi@hznu.edu.cn}\\
 School of Mathematics, Hangzhou Normal University, Hangzhou 311121, China.
  
\end{document}